\documentclass[11pt]{amsart}
\usepackage{amssymb, pstricks,pst-plot} 

\parskip=\smallskipamount
\numberwithin{equation}{section}

\newtheorem{theorem}{Theorem}[section]
\newtheorem{lemma}[theorem]{Lemma}
\newtheorem{corollary}[theorem]{Corollary}
\newtheorem{proposition}[theorem]{Proposition}
\newtheorem{sublemma}[theorem]{Sublemma}

\theoremstyle{definition}
\newtheorem{definition}[theorem]{Definition}
\newtheorem{example}[theorem]{Example}

\newtheorem{remark}[theorem]{Remark}
\newtheorem{question}[theorem]{Question}

\newcommand{\B}{\mathbb{B}}
\newcommand{\C}{\mathbb{C}}

\newcommand{\Z}{\mathbb{Z}}
\renewcommand{\P}{\mathbb{P}}
\newcommand{\R}{\mathbb{R}}
\newcommand{\T}{\mathbb{T}}

\newcommand{\cA}{\mathcal{A}}

\newcommand{\cC}{\mathcal{C}}
\newcommand{\cD}{\mathcal{D}}

\newcommand{\cL}{\mathcal{L}}

\newcommand{\cO}{\mathcal{O}}

\newcommand\wt{\widetilde}
\newcommand\hra{\hookrightarrow}

\newcommand{\nad}[2]{\genfrac{}{}{0pt}{}{#1}{#2}}

\def\ss{\Subset}
\def\di{\partial}
\def\dibar{\bar\partial}
\def\bs{\backslash}

\def\e{\epsilon}
\def\l{\lambda}
\def\d{\delta}

\begin{document}
\title[Strongly pseudoconvex domains as subvarieties]
{Strongly pseudoconvex domains as subvarieties \\ of complex manifolds}
\author{Barbara Drinovec Drnov\v sek \& Franc Forstneri\v c}
\address{Faculty of Mathematics and Physics, University of Ljubljana, 
and Institute of Mathematics, Physics and Mechanics, Jadranska 19, 
1000 Ljubljana, Slovenia}
\email{barbara.drinovec@fmf.uni-lj.si}
\email{franc.forstneric@fmf.uni-lj.si}
\thanks{Research supported by grants P1-0291 and J1-2043-0101, Republic of Slovenia.}

%
%

\date{July 9, 2009} 

%
%
%
%
\begin{abstract}
In this paper we obtain existence and approximation results 
for closed complex subvarieties that are normalized by strongly 
pseudoconvex Stein domains.  Our sufficient condition for 
the existence of such subvarieties 
in a complex manifold $X$ is expressed in terms of 
the Morse indices and the number of positive Levi eigenvalues 
of an exhaustion function on $X$ (Theorem \ref{Main1}). 
Examples show that our conditions cannot be weakened in general. 
We obtain optimal results for subvarieties of this type 
in complements of compact complex submanifolds with 
Griffiths positive normal bundle (Section \ref{subvariety-complements}); 
in the projective case these generalize classical theorems of 
Remmert, Bishop and Narasimhan concerning proper holomorphic maps 
and embeddings to $\C^n =\P^n\bs \P^{n-1}$.
\end{abstract}
\maketitle

{\small \rightline{\em  Dedicated to Edgar Lee Stout}}

%
%
%
%
\section{Introduction}
An interesting and difficult problem in analytic geometry is to describe
the closed complex subvarieties of a given complex 
(or algebraic) manifold $X$. The set of all compact subvarieties  
-- the {\em Douady space} $\cD(X)$ and its close relative, the
{\em cycle space} $\cC(X)$ (the {\em Chow variety} in the 
quasi-projective setting) -- is itself a finite dimensional 
complex analytic space (see \cite{Barlet,Campana,Douady}).
Noncompact subvarieties are in many aspects harder to deal with,
and consequently not as well understood. 

In the present paper we continue the investigation,
begun in \cite{BDF1}, of the existence and plenitude
of subvarieties that arise as proper holomorphic images 
of strongly pseudoconvex Stein domains.
In \cite{BDF1} we analysed the one dimensional case -- 
complex curves normalized by bordered Riemann surfaces.
Here we study  higher dimensional subvarieties of this type 
and obtain optimal results in terms of the Levi geometry and 
the Morse indices of an exhaustion function on the ambient manifold.

Let $X$ be a complex manifold with the complex structure operator 
$J\in \mathrm{End}_\R TX$, $J^2=-I$. 
The {\em Levi form} of a $\cC^2$-function $\rho\colon X\to \R$ is 
\[
	\cL_\rho(x;v)= \frac{1}{4} \langle dd^c\rho, v\wedge Jv\rangle,
	\quad x\in X, \ v\in T_x X,
\]
where $d^c=-J^* \circ d= \mathrm{i}(\dibar -\di)$ is the conjugate differential
defined by $\langle d^c\rho,v\rangle = -\langle d\rho, Jv\rangle$.
We have $dd^c=2\,\mathrm{i} \di\dibar$. 
Choosing local holomorphic coordinates 
$z=(z_1,\ldots,z_n)$ near a point $x\in X$ and writing 
$v=\frac{1}{2}(\eta +\bar \eta)$,
where $\eta = \sum_{j=1}^n \eta_j\frac{\di}{\di z_j}|_x \in T^{1,0}_x X$,
we have
\[
	\cL_\rho(x;v) = 
	\langle \di\dibar \rho(x), \eta \wedge \bar \eta \rangle
	= \sum_{j,k=1}^n \frac{\di^2 \rho(x)}{\di z_j\di \bar z_k} \,\eta_j\,\bar \eta_k.
\]

Our main result is the following.

%
%
%
%
\begin{theorem}
\label{Main1}
Assume that $X$ is an $n$-dimensional complex manifold,
$\Omega$ is an open subset of $X$, 
$\rho\colon \Omega\to (0,+\infty)$ is a smooth Morse function 
whose Levi form has at least $r$ positive eigenvalues at every point 
of $\Omega$ for some $r\le n$, 
and for any pair of real numbers $0<c_1<c_2$ the set 
\[
	\Omega_{c_1,c_2}= \{x\in\Omega\colon c_1\le \rho(x)\le c_2\}
\]
is compact. 
Let $D$ be a smoothly bounded, relatively compact, 
strongly pseudoconvex domain in a 
Stein manifold $S$, and let $f_0\colon\bar D\to X$ be a continuous map 
that is holomorphic in $D$ and satisfies $f_0(bD)\subset \Omega$. If
\begin{itemize}
\item[(a)] $r\ge 2d$, where $d=\dim_{\C} S$, 
\end{itemize}
or if
\begin{itemize}
\item[(b)] $r\ge d+1$ and $\rho$ has no critical points of index
$>2(n-d)$ in $\Omega$,
\end{itemize}
then $f_0$ can be approximated, uniformly on compacts in $D$, 
by holomorphic maps $f\colon D\to X$ such that
$f(z)\in \Omega$ for every $z\in D$ sufficiently close to $bD$,
and 
\[
	\lim_{z\to bD} \rho(f(z))=+\infty.
\]
Moreover, given an integer $k\in\Z_+$, $f$ can be chosen to agree with 
$f_0$ to order $k$ at each point in a given finite set $\sigma\subset D$. 
\end{theorem}

The most interesting case is when $X$ is noncompact, $\Omega$ 
is a union of connected components of $X\bs K$ for some compact subset $K$
of $X$, and $\rho\to +\infty$ along the noncompact ends of $\Omega$. 
Theorem  \ref{Main1} then furnishes
{\em proper holomorphic maps} $f\colon D\to X$ that approximate 
a given map $f_0$ uniformly on compacts in $D$.
A typical situation is $\Omega=\{\rho>0\}$ where $\rho\colon X\to \R$
is an exhaustion function satisfying the stated properties
on $\Omega$.

A $\cC^2$-function $\rho$ on an $n$-dimensional complex manifold $X$
whose Levi form has at least $r$ positive eigenvalues
at every point in an open set $\Omega\subset X$ 
is said to be {\em $(n-r+1)$-convex on $\Omega$} (see \cite{Grauert2}). 
All Morse indices of such function are $\le r+2(n-r)=2n-r$. 
(See Lemma \ref{lemma-normal} below for a quadratic normal form of such a function
at a critical point. The Morse condition can always be achieved by a 
small perturbation of $\rho$ in the fine $\cC^2$-topology, keeping the above 
Levi convexity property of $\rho$.)
Hence condition (a) in Theorem \ref{Main1} implies that all 
Morse indices of $\rho$ in $\Omega$ are $\le 2(n-d)$, and therefore condition (b) 
holds as well. When $d=1$ (i.e., $D$ is a bordered Riemann surface),
conditions (a) and (b) are both equivalent to $r\ge 2$, and in this case 
Theorem \ref{Main1} is essentially the same as \cite[Theorem 1.1]{BDF1}. 
When $d>1$, condition (b) is weaker than (a).

An $n$-dimensional complex manifold $X$ is 
said to be {\em $q$-convex} if it admits an exhaustion
function $\rho\colon X\to \R$ that is $q$-convex on 
$\{\rho>c\}$ for some $c\in\R$;
$X$ is {\em $q$-complete} if $\rho$ can be chosen 
$q$-convex on all of $X$ (see \cite{AG,Grauert2}). 
By approximation we can assume that $\rho$ is 
$\cC^\infty$-smooth. We have the following 
corollary of Theorem \ref{Main1} (a).

%
%
%
%
\begin{corollary}
\label{q-complete}
Let $X$ be an $n$-dimensional complex manifold,
and let $D\Subset S$ be a $d$-dimensional strongly pseudoconvex domain as in 
Theorem \ref{Main1}. Assume that $2d \le n$ and 
$q\in \{1,\ldots, n-2d+1\}$. Then the following hold:

(a) If $X$ is $q$-convex then there exists a proper holomorphic
map $D \to X$. 

(b) If $X$ is $q$-complete then every continuous map  $\bar D\to X$ that 
is holomorphic in $D$ can be approximated, uniformly on compacts 
in $D$, by proper holomorphic maps $D\to X$.
\end{corollary}

Theorem \ref{Main1} and Corollary \ref{q-complete} 
apply to any Stein manifold $D$ 
with compact closure $\bar D$ and smooth strongly 
pseudoconvex boundary $bD$. Indeed, such $D$ is equivalent to a smoothly 
bounded strongly pseudoconvex domain in a Stein manifold, and even in an 
affine algebraic manifold, by a biholomorphism extending smoothly to 
the boundary (see \cite{Catlin,Heunemann3,Stout}).
For the general theory of Stein manifolds we refer to \cite{GR,Ho}.

The image $V=f(D)$ of a proper holomorphic map $f\colon D\to X$ is 
a closed complex subvariety of $X$ (see Remmert \cite{Remmert2}). 
If the generic fiber of $f$ is a single point of $D$ (which is
easily ensured by a suitable choice of the initial map $f_0$),
then $f\colon D\to V$ is a normalization map of the subvariety $V$.

%
%
%
%
Our proof of Theorem \ref{Main1} (see \S\ref{proof})
involves three main analytic techniques.
When $d=\dim D=1$, $D$ is a bordered Riemann surface, 
and in this case Theorem \ref{Main1} essentially
coincides with \cite[Theorem 1.1]{BDF1}. 
The higher dimensional case requires a considerably 
more delicate technique for lifting the boundary of $D$ 
(considered as a subset of $X$ via a map 
$\bar D \to X$) to higher levels of $\rho$.
The main local lifting lemma (see Lemma \ref{MainLemma}) 
employs special holomorphic peak functions that reach their maximum
along certain Legendrian (complex tangential) submanifolds of maximal
real dimension $d-1$ in $bD$. Its proof mainly relies on the 
work of Dor \cite{Dor1} (see also Hakim \cite{Hakim} and Stens\o nes \cite{Sten}).
The idea of using such peak functions goes back to the construction of inner
functions by Hakim and Sibony (see \cite{HakimSibony}) and 
L\o w (see \cite{Low1}); these undoubtedly belong among 
the most intricate and beautiful results in complex analysis.

Each local modification is patched with the previous global map
$\bar D\to X$ by the method of {\em gluing holomorphic sprays} 
developed in \cite{BDF1} (see \S\ref{spray} below).
This technique effectively replaces the $\dibar$-equation
which cannot be used directly in a nonlinear setting.
However, the lemma on gluing of sprays from \cite{BDF1}
depends on the existence of a bounded linear solution operator for
the $\dibar$-equation at the level of $(0,1)$-forms.

To avoid the critical points of $\rho$ (where the 
estimates in the lifting process cannot be controlled) 
we adapt a method that was developed 
(for strongly plurisubharmonic functions) in \cite{ACTA}. 
We first ensure that the boundary of $D$ (considered as a subset of $X$) 
avoids the stable manifold of any critical point of $\rho$; 
this is possible by general position, provided that all
Morse indices of $\rho$ are $\le 2(n-d)$. 
In order to lift $bD$ over the critical level at a critical point $p\in \Omega$ 
we construct a new noncritical function $\tau$, 
with the same Levi convexity properties as $\rho$, 
such that $\{\tau \le 0\}$ contains $\{\rho \le c\}$ for some $c<\rho(p)$,
and it also contains the local stable manifold of $p$ for the gradient flow of $\rho$
(see Lemma \ref{crossing}). Using the lifting procedure 
with $\tau$ we can push $bD$ into $\{\rho>\rho(p)\}$,
and the construction may proceed.

%
%
%
%
In the remainder of this introduction we discuss 
further corollaries and examples related to Theorem \ref{Main1}.

%
%
%
%
\begin{example} 
\label{Counterex}
{\em Condition (b) in Theorem \ref{Main1}
cannot be weakened for any pair of dimensions $1\le d<n$}.
Given integers $1\le d<n$, set $m=n-d+1\in \{2,\ldots,n\}$.
Let $\T^m=\C^m/\varGamma$ be a complex torus of dimension $m$ 
that is not projective-algebraic, and that furthermore 
does not contain any closed complex curves. (Most tori of dimension 
$>1$ are such; for a specific example with $m=2$ see 
\cite[p.\ 222]{Wells}.) Set
\begin{equation}
\label{exampleX}
	X =\T^m\bs \{p\} \times \C^{n-m}= \T^m\bs \{p\} \times \C^{d-1}.
\end{equation}
Choose an exhaustion function $\tau \colon \T^m\bs \{p\}\to \R$ 
that equals $|y-y(p)|^{-2}$ in some local holomorphic coordinates 
$y$ on $\T^m$ near $p$.  The exhaustion function $\rho(y,w)= \tau(y)+ |w|^2$ on $X$ 
has no critical points in a deleted neighborhood of $p$, 
and its Levi form has $1+n-m=d$ positive eigenvalues near $\{p\}\times\C^{d-1}$.
Thus $X$ satisfies condition (b) in Theorem  \ref{Main1} for 
domains $D$ of dimension $<d$, but not for domains of dimension~$\ge d$.

We claim that no $d$-dimensional Stein manifold $D$ admits
a proper holomorphic map to the manifold $X$ (\ref{exampleX}).
Indeed, suppose that $f\colon D\to X$ is such a map.
Let $\pi\colon X\to \C^{d-1}$ denote the projection $\pi(y,w)=w$
onto the second factor. Consider the holomorphic map
$\pi\circ f\colon D\to\C^{d-1}$.
By dimension reasons there exists a point $w\in \C^{d-1}$ 
for which the fiber $\Sigma=\{z\in D\colon \pi(f(z))=w\}$ 
is a subvariety of positive dimension in $D$.
Since $D$ is Stein, $\Sigma$ contains a one dimensional subvariety $C$, 
and $f(C)$ is then a closed complex curve in $\T^m\bs \{p\}\times \{w\}$.
Since a point is a removable singularity for positive dimensional 
analytic subvarieties \cite{RS}, it follows that $\overline{f(C)}$
is a nontrivial closed complex curve in $\T^m\times \{w\}$, 
a contradiction.

Interestingly enough, the manifold $X$ (\ref{exampleX}) admits plenty
of {\em nonproper} holomorphic maps $S\to X$ from any Stein manifold
$S$.  Indeed, $X$ enjoys the following {\em Oka property} 
(see \cite[Corollary 1.5 (ii)]{ANN}): 

{\em Any continuous map $f_0\colon S\to X$ from a Stein manifold $S$
is homotopic to a holomorphic map $f\colon S\to X$;
if in addition $f_0$ is holomorphic in a neighborhood of a 
compact $\cO(S)$-convex subset $K\subset S$, then $f$ can be chosen 
to approximate $f_0$ as close as desired uniformly on $K$.}

This shows that properness of a holomorphic map
$f\colon D\to X$ is a very restrictive condition 
irrespectively of the codimension $\dim X - \dim D$.
\qed \end{example}

%
%
%
%
Theorem \ref{Main1} gives interesting information on the
existence of proper holomorphic maps of strongly pseudoconvex
domains into complements of certain complex submanifolds.
For example, if $A$ is a compact complex submanifold of 
complex codimension $q$ in a projective space $X=\P^n$, 
then $\Omega =\P^n\bs A$ admits a $q$-convex exhaustion function without 
critical points close to $A$ (Barth \cite{Barth}; 
the manifold  $\P^n\bs \P^{n-q}$ is even $q$-complete.)
Thus condition (b) in Theorem \ref{Main1} holds when $r=n-q+1 > \dim D$
or, equivalently, $\dim D\le \dim A$. This gives the first part
of the following corollary; for the second part we apply
a result of Schneider (see \cite[Corollary 2]{Schneider}).

\begin{corollary}
\label{proj-minus}
If $A$ is a compact complex submanifold of $\P^n$ 
then every smoothly bounded, relatively compact, 
strongly pseudoconvex Stein domain $D$ 
of dimension $\dim D\le \dim A$ 
admits a proper holomorphic map $D\to \P^n\bs A$.
In particular, if $\dim D<n$ then $D$ admits a proper holomorphic 
map into the complement $\P^n\bs A$ of any nonsingular complex
hypersurface $A$ in $\P^n$.
The analogous conclusion holds for maps $D\to X\bs A$, where $A$ is a 
compact complex submanifold of a complex manifold $X$ 
with Griffiths positive normal bundle $N_{A|X}$.
\end{corollary}

Corollary \ref{proj-minus} generalizes Bishop's theorem 
(see \cite{Bishop}) on the existence of proper holomorphic maps
$D\to \C^n=\P^n\bs \P^{n-1}$ for $n>\dim D$. 
While Bishop's theorem holds for any Stein manifold $D$
of dimension $<n$, in the general situation considered here
one must restrict to Kobayashi hyperbolic domains since 
the complement $\P^n\bs A$ of a generic hypersurface 
$A\subset\P^n$ of sufficiently high degree is hyperbolic.

The conclusion of Corollary \ref{proj-minus} fails when 
$\dim D >\dim A$. Indeed, the closure of $V=f(D)$ in $\P^n$ 
would be a closed complex subvariety of $\P^n$
by the Remmert-Stein theorem (see \cite[p.\ 354]{RS,GRemmert}),
hence $f(D) =\overline V \bs A$ would be quasi-projective algebraic
(the difference of two closed projective varieties). 
This is clearly impossible. It is easily seen that a 
proper holomorphic map $f\colon D\to X\bs A$ 
cannot extend continuously (as a map to $X$) 
to any boundary point of $D$.

Our next corollary generalizes classical results of Remmert 
\cite{Remmert1}, Bishop \cite{Bishop},
and Narasimhan \cite{Nar1} on immersions and embeddings 
of strongly pseudoconvex domains into Euclidean spaces, 
as well as results of Dor \cite{Dor0,Dor1} where the target 
manifold $X$ is a domain of holomorphy in $\C^n$.

%
%
%
%
\begin{corollary}
\label{cor1}
Assume that $D$ is a smoothly bounded, relatively compact,
strongly pseudoconvex domain in a Stein manifold $S$, 
$X$ is a Stein manifold, and $f_0\colon\bar D\to X$ is a continuous map 
that is holomorphic in $D$. 
\begin{itemize}
\item[(i)] If $\dim X\ge 2\dim D$, then $f_0$ can be approximated 
uniformly on compacts in $D$ by proper holomorphic immersions $D\to X$.
\item[(ii)] If $\dim X\ge 2\dim D +1$, then $f_0$ can be approximated 
uniformly on compacts in $D$ by proper holomorphic embeddings $D\hra X$.
\end{itemize}
\end{corollary}

Corollary \ref{cor1} is a consequence of Theorem \ref{Main1}
(condition (a) holds since a Stein manifold $X$ is $1$-complete), 
except for the claim that $f$ can be chosen 
an immersion (resp.\ an embedding). The latter conditions
are easily built into the construction by applying a 
general position argument at every step of the inductive process. 

Further results on holomorphic immersions and embeddings in $X=\C^n$ 
can be found in \cite{EG,FW,FIKP,Prezelj,Sch,Wo1,Wo2}; for embeddings into 
special domains such as balls and polydiscs see also 
\cite{DorBalls,TAMS,Glob1,Hakim,Low2,Sten}.

%
%
%
%
Assume now that $X$ is a quasi-projective algebraic manifold.
We shall see that in this case every subvariety $V=f(D)\subset X$,
obtained by the proof of Theorem \ref{Main1}, is a limit of domains
contained in algebraic varieties in $X$ and normalized by $D$. 

By a theorem of Stout \cite{Stout} (see also \cite{DLS,LM})
we can assume that $D$ in Theorem \ref{Main1} is a smoothly bounded, 
strongly pseudoconvex, Runge domain 
in an affine algebraic manifold $S\subset\C^N$
of pure dimension $d$. A holomorphic map $f$ from an open set $U\subset S$ to 
a quasi-projective algebraic variety $X$ is said to be {\em Nash algebraic} 
(see Nash \cite{Nash}) if its graph 
\[
	G_f=\{(z,f(z)) \in S\times X\colon z\in U\}
\]
is contained in a pure $d$-dimensional algebraic subvariety of 
$S\times X$. We then have the following result 
(c.f.\ \cite[Corollary 1.2]{BDF1} for $d=1$).

%
%
%
%
\begin{corollary}
\label{algebraic}
Assume that $X$ is a quasi-projective algebraic manifold,
$\rho\colon X\to \R$ is a smooth exhaustion 
function that satisfies one of the conditions in 
Theorem \ref{Main1} on the set $\Omega=\{x\in X\colon \rho(x)>0\}$,
and $D\Subset S$ is a smoothly bounded, 
strongly pseudoconvex Runge domain in an affine algebraic
manifold $S$. Given a map $f_0\colon \bar D \to X$ 
as in Theorem \ref{Main1}, with $f_0(bD)\subset \Omega$,
there is a sequence of 
Nash algebraic maps $f_j\colon U_j\to X$,
defined in open sets $\bar D\subset U_j\subset S$, 
such that $f_j(bD)\subset \Omega$, 
\[
	\lim_{j\to\infty} \, \bigl( \inf\{\rho\circ f_j(z) \colon z\in bD\}\bigr)  \to +\infty,
\]
and the sequence $f_j|_D$ converges to a proper holomorphic map $f\colon D\to X$
as $j\to\infty$. Furthermore, $f$ can be chosen to approximate $f_0$ as close 
as desired uniformly on a given compact subset of $D$.
\end{corollary}

The image $f_j(U_j)$ of the Nash algebraic map $f_j$ in 
Corollary \ref{algebraic} is contained in a pure 
$d$-dimensional algebraic subvariety $\varGamma_j$ of $X$
(the projection to $X$ of an algebraic subvariety in $S\times X$
containing the graph of $f_j$). As $j\to\infty$, the domains
$f_j(D)\subset \varGamma_j$ converge to the subvariety 
$f(D) \subset X$, while their boundaries $f_j(bD)$ tend to infinity in $X$.

Corollary \ref{algebraic} is seen exactly as \cite[Corollary 1.2]{BDF1} 
by combining the proof of Theorem \ref{Main1} with the  
approximation theorems of Demailly, Lempert and Shiffman 
(see \cite[Theorem 1.1]{DLS}) and Lempert (see \cite[Theorem 1.1]{Lempert}).

%
%
%
%

Here is another natural question:

\begin{question}
When is a continuous map $D\to X$ from a 
strongly pseudoconvex Stein domain $D$ to a complex manifold $X$
homotopic to a proper holomorphic map?
\end{question} 

The following result in this direction generalizes 
Corollaries 1.5 and 1.6 in \cite{BDF1} which concern 
the one dimensional case; the proofs given there
also apply in our case by using Theorem \ref{Main1}. 
The Oka property was defined in 
Example \ref{Counterex}; for more details see \cite{ANN}.

\begin{corollary}
\label{cor3}
Let $D\Subset S$ be a smoothly bounded, 
strongly pseudoconvex domain in a 
$d$-dimensional Stein manifold $S$, and let $X$ be a complex manifold 
of dimension $n\ge 2d$ that is $(n-2d+1)$-complete.
Let $J_S$ (resp.\ $J_X$) denote the complex structure operator on $S$
(resp.\ on $X$).
\begin{itemize} 
\item[(i)] If $d\ne 2$ then for every continuous map $f_0\colon\bar D\to X$
there exists a Stein structure $\wt J_S$ on $S$ that is homotopic to $J_S$
and such that $D$ is strongly $\wt J_S$-pseudoconvex, 
and there exists a proper $(\wt J_S,J_X)$-holo\-morphic map 
$f\colon D\to X$ homotopic to $f_0|_D$.
\item[(ii)] If $d=2$ then the conclusion (i) holds after changing the 
$\cC^\infty$ structure on $S$ (i.e., the new Stein structure 
$\wt J_S$ may be exotic). 
\item[(iii)] If $X$ enjoys the Oka property  then every
continuous map $\bar D\to X$ is homotopic to a proper 
$(J_S,J_X)$-holomorphic map $D\to X$.
\end{itemize}
\end{corollary}

For further results see Theorem \ref{complements}, Theorem \ref{homogeneous}
and Corollary \ref{projective}.

%
%
%
%
{\em Organization of the paper.}
In \S \ref{normal} and \S \ref{critical}
we analyse the behavior of a $q$-convex function near a
Morse critical point. 
In \S\ref{spray} we recall the relevant results from \cite{BDF1} 
on the theory of holomorphic sprays.
Theorem \ref{Main1} is proved in \S\ref{proof}. 
In \S\ref{vector-bundles} we recall the notions
of Griffiths positivity and signature of a Hermitian
holomorphic vector bundle, as well as their connection with the 
Levi convexity properties. This information is used in 
\S\ref{subvariety-complements}
where we study the existence of subvarieties as in 
Theorem \ref{Main1} in complements of certain compact 
complex submanifolds.

%
%
%
%
%
%
\section{Quadratic normal form for critical points of $q$-convex functions}
\label{normal}
In this section we describe  a quadratic normal form of a $q$-convex
function $\rho$ at a nondegenerate critical point $p$.
In the following section this normal form will be used
in the construction of a $q$-convex function $\tau$ 
that allows us to pass the critical level $\{\rho=\rho(p)\}$
by applying the noncritical case of our lifting construction 
with $\tau$ (instead of $\rho$).  

Since our considerations are completely local,
we assume that $\rho$ is a real valued $\cC^2$-function in an 
open neighborhood of the origin in $\C^n$, with a 
nondegenerate (Morse) critical point at $0$, and $\rho(0)=0$.
Suppose that $\rho$ is $q$-convex at $0$ for some 
$q\in\{1,2,\ldots,n+1\}$; this means that its Levi form
$\cL_\rho(0)$ has at least $r=n-q+1$ positive eigenvalues
(the remaining $s=q-1$ eigenvalues can be of any sign). 
By a complex linear change of coordinates on $\C^n$ we can achieve
that the subspace $\C^r\times \{0\}^s$ is spanned
by (some of the) eigenvectors corresponding to the positive eigenvalues 
of $\cL_\rho(0)$ and that $0$ is a Morse critical point of
$\rho(\cdotp,0)$. We denote the coordinates
on $\C^n=\C^r\times \C^s= \C^r \times \R^{2s}$ by $z=(\zeta,u)$,
where $\zeta=x+\mathrm{i}y\in \C^r$ $(x,y\in\R^r)$ and $u\in\R^{2s}$.
By shrinking the domain of $\rho$ to a sufficiently 
small polydisc $P=P^r \times P^s \subset \C^n$ around $0$ 
we can assume that the function $\zeta\to \rho(\zeta,u)$ is strongly
plurisubharmonic on $P^r$ for each fixed $u \in P^s$.

Lemma 3 from \cite{HW}, applied to the strongly plurisubharmonic
function $\zeta \to \rho(\zeta,0)$, 
gives a complex linear change of coordinates on $\C^r$ 
and a number $k\in \{0,1,\ldots, r\}$ such that,
in the new coordinates, we have
\[
	\rho(\zeta,0) = \sum_{j=1}^r (\delta_j x_j^2 + \lambda_j y_j^2)  +o(|\zeta|^2),
\]
where $\lambda_j>1$, $\delta_j=-1$ for $j=1,\ldots,k$,
and $\lambda_j\ge 1$, $\delta_j=+1$ for $j=k+1,\ldots,r$. 
Note that $k$ is the Morse index of $\rho(\cdotp,0)$ at 
$\zeta=0$.

Writing $x'=(x_1,\ldots,x_k)$ and $x''=(x_{k+1},\ldots,x_r)$ we obtain
\[
	 \rho(\zeta,0) =  - |x'|^2 + |x''|^2 + \sum_{j=1}^r \lambda_j y_j^2 + o(|\zeta|^2). 
\]

Consider the full second order Taylor expansion of $\rho$ at $0\in\C^n$:
\[
	\rho(z)= \rho(\zeta,u)= \rho(\zeta,0) + \sum_{j=1}^{2s} u_j a_j(x,y) +
		\sum_{i,j=1}^{2s} c_{ij}\, u_i u_j  + o(|z|^2).
\]
Here $a_j(x,y)= \sum_{l=1}^r \alpha_{jl} x_l+\beta_{jl} y_l$ 
are real-valued linear functions on $\C^r=\R^{2r}$ 
and $c_{ij}=c_{ji}$ are real constants. 

Our next aim is to remove the mixed terms $u_j a_j(x,y)$ by using 
a shear of the form $(\zeta,u)\mapsto (\zeta+h(u),u)$ for a
suitable $\R$-linear map $h\colon \R^{2s}\to \C^r$; such transformation 
is holomorphic (indeed, affine linear) in the $\zeta$-coordinates, 
and hence it preserves (strong) $\zeta$-plurisubharmonicity.
To find such $h$, we consider the critical point equation 
$\di_\zeta \rho^{(2)}(\zeta,u)=0$, where $\rho^{(2)}$ denotes 
the 2nd order homogeneous polynomial of $\rho$:
\[
		\frac{\di \rho^{(2)}}{\di x_i}(\zeta,u) = 
		2\delta_i x_i + \sum_{j=1}^{2s} u_j \alpha_{ji}=0; \quad
		\frac{\di \rho^{(2)}}{\di y_i}(\zeta,u) = 
		2\lambda_i y_i + \sum_{j=1}^{2s} u_j \beta_{ji}=0.
\] 
This system has a unique (linear) solution $\zeta=x+iy=h(u)$, and
the quadratic map $\zeta\to \rho^{(2)}(\zeta+h(u),u)$ has a unique critical
point at $\zeta=0$ for every $u$. Writing 
$\rho(\zeta,u)= \wt \rho(\zeta+h(u),u)$, the function
$\wt\rho$ is of the same form as $\rho$, but with $a_j(x,y)=0$ for all 
$j=1,\ldots, 2s$. We drop the tilde and denote 
the new function again by $\rho$. 

The classical theorem of Sylvester furnishes an $\R$-linear    
transformation of the $u$-coordinates which puts 
$\sum_{i,j=1}^{2s} c_{ij}\, u_i u_j$ into a normal form
$-|u'|^2 + |u''|^2$, where $u'=(u_1,\ldots,u_m)$ and 
$u''=(u_{m+1},\ldots,u_{2s})$ for some $m\in\{0,1,\ldots,2s\}$. 
This gives $\rho(\zeta,u)=\wt\rho(\zeta,u)+ o(|\zeta|^2+|u|^2)$ where 
\begin{equation}
\label{normal-form}
	\wt\rho(\zeta,u)= - |x'|^2  -|u'|^2 + |x''|^2 + 
					|u''|^2 + \sum_{j=1}^r \lambda_j y_j^2, 
\end{equation}	
$\lambda_j>1$ for $j=1,\ldots,k$, and $\lambda_j\ge 1$
for $j=k+1,\ldots,r$. 
We shall say that (\ref{normal-form})
is a {\em quadratic normal form} for critical points of $q$-convex functions.
Note that $k+m$ is the Morse index of $\rho$ (or $\wt\rho$) at $0$.

We summarize the above discussion in the following lemma;
for the strong\-ly pseudoconvex case see \cite[Lemma 2.5]{HL2}.

%
%
%
%
\begin{lemma}
\label{lemma-normal}
{\em (Quadratic normal form of a $q$-convex critical point)}
Assume that $X$ is an $n$-dimensional complex manifold
and that $\rho\colon X\to \R$ is a $\cC^2$-function with
a nondegenerate critical point at $p_0 \in X$.
If $\rho$ is $q$-convex at $p_0$ for some $q\in \{1,\ldots,n+1\}$ 
then there exist 
\begin{itemize}
\item[(i)]  a local holomorphic coordinate map 
$z=(\zeta,w) \colon U\to \C^r\times \C^s$ 
on an open neighborhood $U\subset X$ of $p_0$, with $z(p_0)=0$,
$r=n-q+1$ and $s=q-1$, 
\item[(ii)] a change of coordinates $\psi(z)=\psi(\zeta,w)=(\zeta+h(w),g(w))$
on $\C^n$ that is $\R$-linear in $w\in\C^{s}=\R^{2s}$, and
\item[(iii)] a normal form $\wt\rho(\zeta,u)$ of type 
(\ref{normal-form}), 
\end{itemize} 
such that, setting $\phi(p)=  \psi(z(p)) \in \C^n$ for $p\in U$, we have
\[
	\rho(p)= \rho(p_0) + \wt\rho( \phi(p)) + o(|\phi(p)|^2),\quad p\in U. 
\]
Furthermore, we can approximate $\rho$ as close as desired in the 
$\cC^2$-topology by a $q$-convex function $\rho'$ that agrees with $\rho$  
outside of $U$ and has a {\em nice critical point} at $p_0$, 
in the sense that $\rho'(p)= \rho(p_0) + \wt\rho( \phi(p))$ near $p_0$.
If $\rho$ is $\cC^r$-smooth for some $r\in\{2,3,\ldots,\infty\}$
then $\rho'$ can also be chosen $\cC^r$-smooth.
\end{lemma}

\begin{proof}
Everything except the claim in the penultimate sentence has 
been proved  above. The latter is seen by taking  
\[
	\rho'(p) = \rho(p_0) + \wt\rho( \phi(p)) + 
	\chi\bigl( \e^{-1}\phi(p)\bigr) \, o(|\phi(p)|^2),
\]
where $\chi\colon \C^n\to[0,1]$ is a smooth function
that equals zero in the unit ball $\B\subset \C^n$, and equals one
outside of $2\B$. When $\e>0$ decreases to zero, the $\cC^2$-norm of the 
last summand tends to zero uniformly on $U$. 
\end{proof}

%
%
%
%
\section{Crossing a critical level of a $q$-convex function}
\label{critical}
Let $p_0 \in X$ be a {\em nice critical point}
of a $\cC^2$-function $\rho\colon X\to \R$ that is $q$-convex near 
$p_0$ and satisfies $\rho(p_0)=0$ (see Lemma \ref{lemma-normal}).
Choose a neighborhood $U\subset X$ of $p_0$ and a coordinate map 
$\phi \colon U\to P$ onto a polydisc $P\Subset \C^n$ 
as in Lemma \ref{lemma-normal} such that the function 
$\wt \rho=\rho\circ\phi^{-1} \colon P\to \R$ 
is a $q$-convex normal form (\ref{normal-form}). Set
$
	Q(y,x'',u'')= \sum_{j=1}^r \lambda_j y_j^2 + |x''|^2 + |u''|^2;
$
hence 
\begin{equation}
\label{simplified}
		\wt\rho(x+\mathrm{i} y,u) = -|x'|^2 - |u'|^2 + Q(y,x'',u'').
\end{equation}		
Let $c_0\in (0,1)$ be chosen sufficiently small such that
\[
		\{(x+\mathrm{i} y,u)\in\C^r\times \R^{2s} \colon 
			 |x'|^2 + |u'|^2 \le c_0,\ Q(y,x'',u'') \le 4c_0\} 
			 \subset P.
\]
Set 
\[
	\wt E=\{(x+\mathrm{i} y,u)\in\C^r\times \R^{2s} \colon |x'|^2 + |u'|^2 \le c_0,\ 
	 			 y=0,\ x''=0,\ u''=0\}.
\] 
Its preimage $E=\phi^{-1}(\wt E)\subset U$ is an embedded disc 
of dimension $k+m$ (the Morse index of $\rho$ at $p_0$) that is attached
from the outside to the sublevel set $\{\rho\le -c_0\}$ along
the sphere $bE\subset \{\rho = -c_0\}$. In the metric on $U$,
inherited by $\phi$ from the standard metric in $\C^n$, $E$ is
the (local) stable manifold of $p_0$ for the gradient flow of $\rho$.

The following lemma generalizes \cite[Lemma 6.7]{ACTA}
to $q$-convex functions. (The cited lemma applies to $q=1$, 
that is, to a strongly plurisubharmonic function $\rho$.)

%
%
%
%
\begin{lemma} 
\label{crossing}
{(Notation and assumptions as above.)} 
Assume that $\rho$ is $q$-convex in the set 
$K_{c_0}= \{p\in X\colon -c_0\le \rho(p) \le 3c_0\} \Subset X$
and that $p_0$ is the only critical point of $\rho$ in $K_{c_0}$.
Assume that a normal form of $\rho$ at $p_0$ is given 
by (\ref{normal-form}), where $k\in \{1,\ldots,r\}$ and 
$\lambda=\min \{\l_1,\ldots,\l_k\}>1$.  
Given a number $t_0$ with $0< t_0 < (1-\frac{1}{\lambda})^2c_0$,
there is a $\cC^2$-function $\tau \colon \{\rho \le 3c_0\} \to\R$
enjoying the following properties (see Figure \ref{Tau}):
\begin{itemize}
\item[(i)]   $\{\rho\le -c_0\} \cup E \subset \{\tau\le 0\} \subset \{\rho\le -t_0\}\cup E$, 
\item[(ii)]  $\{\rho \le c_0\} \subset \{\tau \le 2c_0\} \subset \{\rho< 3c_0\}$,
\item[(iii)] $\tau$ is $q$-convex at every point of $K_{c_0}$, and
\item[(iv)]  $\tau$  has no critical values in $(0,3c_0) \subset \R$.
\end{itemize}
If $\rho$ is $\cC^r$-smooth for some $r\in\{2,3,\ldots,\infty\}$
then $\tau$ can also be chosen smooth of class $\cC^r$.
\end{lemma}

\begin{figure}[ht]
\psset{unit=0.6cm, xunit=1.5, linewidth=0.7pt} 
\begin{pspicture}(-4.5,-4.3)(4.5,4.3)

%
%
%
%
\pscustom[fillstyle=solid,fillcolor=lightgray,linestyle=none]  
{
\pscurve(-3,-1.5)(-2.5,-0.8)(0,-0.3)(2.5,-0.8)(3,-1.5) 
\psline[linestyle=dashed,linewidth=0.2pt](3,-1.5)(3,1.5)
\pscurve[liftpen=1](3,1.5)(2.5,0.8)(0,0.3)(-2.5,0.8)(-3,1.5)
\psline[linestyle=dashed,linewidth=0.2pt](-3,1.5)(-3,-1.5)
}

\pscustom[fillstyle=solid,fillcolor=lightgray]
{\pscurve[liftpen=1](5,4)(3,1.5)(2.5,0)(3,-1.5)(5,-4)          
}

\pscustom[fillstyle=solid,fillcolor=lightgray]
{
\pscurve[liftpen=1](-5,4)(-3,1.5)(-2.5,0)(-3,-1.5)(-5,-4)      
}

%
%
%
\psline(-2.5,0)(2.5,0)                                        
\psecurve(5,4)(3,1.5)(2.5,0.8)(0,0.3)(-2.5,0.8)(-3,1.5)(-5,4) 
\psecurve(5,-4)(3,-1.5)(2.5,-0.8)(0,-0.3)(-2.5,-0.8)(-3,-1.5)(-5,-4)  

%
%
%
\pscurve[linestyle=dotted,linewidth=1pt](4,4)(2,1.5)(1.5,0)(2,-1.5)(4,-4)             \pscurve[linestyle=dotted,linewidth=1pt](-4,4)(-2,1.5)(-1.5,0)(-2,-1.5)(-4,-4)            

\rput(0,3){$\{\rho=-t_0\}$}
\psline[linewidth=0.2pt]{->}(1,3)(3.05,3)
\psline[linewidth=0.2pt]{->}(-1,3)(-3.05,3)

%
%
%

\psline[linewidth=0.2pt]{->}(0,-2.2)(0,-0.3)
\psline[linewidth=0.2pt]{<-}(-3.7,-2)(-0.3,-2.6)
\psline[linewidth=0.2pt]{->}(0.35,-2.6)(3.7,-2)
\rput(0.05,-2.6){$\Omega_c$}

\psline[linewidth=0.2pt]{->}(-1,1.7)(-1,0.05)
\rput(-1,2){$E$}

\psdot(0,0)
\rput(0,1.2){$p_0$}
\psline[linewidth=0.2pt]{->}(0,1)(0,0.05)

\rput(3.8,0){$\{\rho<c-t_1\}$}
\rput(-3.8,0){$\{\rho<c-t_1\}$}

\end{pspicture}
\caption{The set $\Omega_c=\{\tau<c\}$}
\label{Tau}
\end{figure}
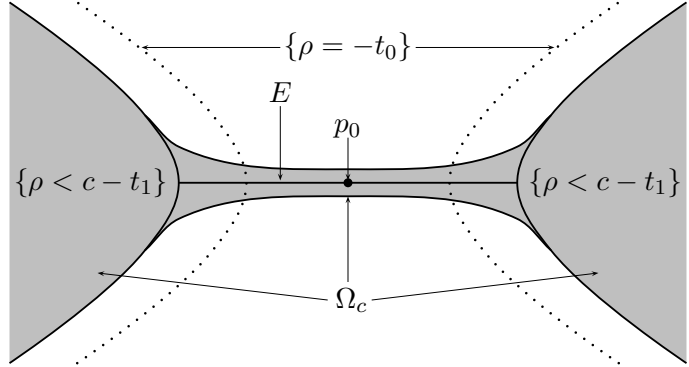

Each sublevel set $\Omega_c=\{\tau < c\}$ for $c\in (0,2c_0)$ 
is a domain with $\cC^2$ strongly $q$-convex 
boundary that contains $\{\rho\le -c_0\} \cup E$, and the latter 
set is a strong deformation retract of $\Omega_c$ (see Figure \ref{Tau}).
As $c$ decreases to $0$, the sets $\Omega_c \cap \{\rho\ge -t_0\}$ 
decrease to the disc $E'= E\cap \{\rho\ge -t_0\}$. Finally, 
the domain $\Omega_{2c_0}$ contains the set $\{\rho<c_0\}$.

\begin{proof}
In \cite[proof of Lemma 6.7, p.\ 178]{ACTA} the second named author
constructed a smooth convex increasing function $h\colon \R \to [0,+\infty)$ 
enjoying  the following properties (see Figure \ref{h}):
\begin{itemize}
\item[(i)]   $h(t)=0$ for $t\le t_0$, 
\item[(ii)]  for $t\ge c_0$ we have $h(t)=t - t_1$ with $t_1=c_0 - h(c_0) \in (t_0,c_0)$,
\item[(iii)] for $t_0\le t\le c_0$ we have $t-t_1 \le h(t) \le t - t_0$, and
\item[(iv)]  for all $t\in\R$ we have $0\le \dot h(t) \le 1$ and $2t\ddot h(t) + \dot h(t) < \lambda$.
\end{itemize}

\begin{figure}[ht]
\psset{unit=0.6cm} 
\begin{pspicture}(-1,-1)(12,7)

%
%
\psline[linewidth=0.4pt]{->}(0,0)(11,0)
\psline[linewidth=0.4pt]{->}(0,0)(0,6)

\psline[linewidth=0.8pt](6,2)(9.5,5.5)   		
\psline[linestyle=dotted](4,0)(6,2)                           
\psline[linestyle=dotted](6,0)(6,2)                           
\psecurve[linewidth=0.8pt](-2,2)(2,0)(6,2)(8,4.5)   
\psline[linewidth=0.8pt](0,0)(2,0)          
\psdots(2,0)(4,0)(6,0)(0,0)                										
\psline[linestyle=dotted](2,0)(7.5,5.5)                       

\rput(2,-0.5){$t_0$}
\rput(4,-0.5){$t_1$}
\rput(6,-0.5){$c_0$} 
\rput(11.5,0){$t$}
\rput(-0.5,-0.5){$(0,0)$}

\rput(8,2){$h$}
\psline[linewidth=0.2pt]{->}(7.7,2.3)(7,3)

\rput(4,4){$t-t_0$}
\psline[linewidth=0.2pt]{->}(4,3.5)(4.7,2.8)

\end{pspicture}
\caption{The function $h$}
\label{h}
\end{figure}
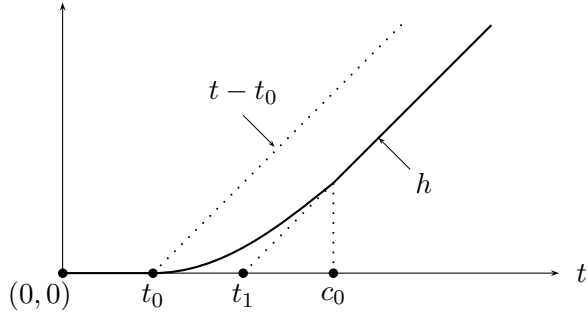

Given such $h$ we define a function 
$\wt \tau\colon \C^n=\C^r\times \R^{2s} \to\R$ by
\begin{equation}
\label{tau}
		\wt\tau(\zeta,u)= - h\bigl( |x'|^2 +|u'|^2\bigr) +  Q(y,x'',u'').
\end{equation}
Its critical locus is $\{|x'|^2+|u'|^2\le t_0,\ x''=0,\ y=0,\ u''=0\} \subset \wt E$
and the corresponding critical value is zero.  
From the property (iv) of $h$ and  \cite[Lemma 6.8]{ACTA} 
we see that $\wt \tau(\cdotp,u)$ is strongly plurisubharmonic
on $\C^r$ for every fixed value of $u\in\R^{2s}$. 
The cited lemma applies directly when $u'=0$; in general we consider the 
translated function $h_c(t)= h(t+c)$ with $c=|u'|^2>0$; 
we have $\dot h_c(t) =\dot h(t+c) \le 1< \l$ and
$2t\ddot h_c(t) + \dot h_c(t)\le 2(t+c)\ddot h(t+c) + \dot h(t+c)  < \lambda$
(we used $\ddot h\ge 0$ and the property (iv) of $h$). 
Lemma 6.8 in \cite{ACTA} now gives the desired conclusion
concerning the function $\zeta\to \wt \tau(\zeta,u)=-h_c(|x'|^2) + Q(y,x'',u'')$,
where $c=|u'|^2$. 

Comparing the definitions of $\wt \rho$ (\ref{simplified}) and $\wt \tau$ (\ref{tau}),
and taking into account the properties of $h$, we see that the 
following conditions hold:
\begin{itemize}
\item[(a)] $\wt \rho \le \wt \tau \le \wt \rho+t_1$ (since $t-t_1\le h(t)\le t$ for all $t\ge 0$), 
\item[(b)] $\wt \rho+ t_0 \le \wt \tau$ on the set $\{|x'|^2+|u'|^2 \ge t_0\}$ (from (ii) and (iii)), and
\item[(c)] $\wt \tau = \wt\rho + t_1$ on the set $\{|x'|^2 +|u'|^2 \ge c_0\}$ (from (ii)).
\end{itemize}

Let $V= \{p\in X\colon \rho(p) \le 3c_0\}$. We define a function 
$\tau \colon V \to\R$ by 
\[	
	\tau= \wt\tau\circ \phi \ \; \text{on}\ U\cap V, \quad  
  \tau = \rho + t_1 \ \text{on}\  V\bs U. 
\]  
Property (c) implies that both definitions of $\tau$ agree on the set
\[
	\{p\in U\cap V \colon |x'(p)|^2 + |u'(p)|^2 \ge c_0\}.
\]
(Here $x'(p)$ and $u'(p)$ denote the corresponding components of $\phi(p) \in\C^n$.) 
Since $\{p\in U\cap V \colon |x'(p)|^2 + |u'(p)|^2 \le c_0\}
\subset \{p\in U \colon |x'(p)|^2 + |u'(p)|^2 \le c_0,\ Q(y(p), x''(p),u''(p))\le 4c_0\}$ and the latter set is compactly contained in $U$, 
we see that $\tau$ is well defined on $V$. 
The stated properties now follow immediately. 
In particular, since $\wt \tau$ is strongly
plurisubharmonic on $\C^r \times\{u\} \subset\C^n$ $(u\in\R^{2s})$
and the coordinate map $\phi\colon U\to P$ is holomorphic on 
$\Sigma_u=\phi^{-1}(\C^r \times\{u\}) \subset U$, the restriction
of $\tau$ to each $r=(n-q+1)$-dimensional complex submanifold $\Sigma_u$ 
of $U$ is strongly plurisubharmonic. Since these submanifolds 
form a smooth nonsingular foliation of $U$ with holomorphic leaves, 
$\tau$ is $q$-convex in $U\cap V$, while on $V\bs U$ it 
is just a translate of $\rho$ by a constant.  
\end{proof}

\begin{remark}
In Lemma \ref{crossing} we exclude the case $k=0$ when 
$\rho$ has a local minimum at the critical point $p_0$ 
in the Levi-positive directions. This case need not be considered
in the proof of Theorem \ref{Main1} since the boundary of $D$
in $X$ cannot approach such a point from below during
the lifting process.
\end{remark}

%
%
%
%
\section{Holomorphic sprays}
\label{spray}
In the proof of Theorem \ref{Main1} we use sprays of maps
to globalize
local corrections made near a small part of the boundary. 
For this purpose we recall from \cite{BDF1,BDF2,FFAsian} 
the relevant results concerning holomorphic sprays,
adjusting them to the applications in this paper.

%
%
%
%
\begin{definition}
\label{Spray}
Let $\ell\ge 2$,  $r\in\{0,\ldots,\ell\}$ and $k\in\Z_+$ be integers.
Assume that $X$ is a complex manifold, $D$ is a relatively compact 
strongly pseudoconvex domain with $\cC^\ell$ boundary in a Stein manifold $S$,
and $\sigma$ is a finite set of points in $D$.
A {\em spray of maps of class $\cA^r(D)$ with the 
exceptional set $\sigma$ of order $k$}
(and with values in $X$) is a map $f\colon \bar D\times P\to X$, where $P$ 
(the {\em parameter set} of the spray) is an open subset of 
a Euclidean space $\C^m$ containing the origin, 
such that the following hold: 
\begin{itemize}
\item[(i)]   $f$ is holomorphic on $D\times P$ and of class $\cC^r$
on $\bar D \times P$, 
\item[(ii)] the maps $f(\cdotp,0)$ and $f(\cdotp,t)$ agree 
on $\sigma$ up to order $k$ for $t\in P$, and
\item[(iii)]  for every $z\in \bar D\bs \sigma$ and $t\in P$  the map
\[
	\di_t f(z,t) \colon T_t \C^n =\C^n  \to T_{f(z,t)} X
\]
is surjective (the {\em domination property}).
\end{itemize}
We shall call $f_0=f(\cdotp,0)$ the {\em core} 
(or {\em central}\/) map of the spray $f$. 
\end{definition}

The following lemma is essentially \cite[Lemma 4.2]{BDF1} 
for the case of sprays of maps. As it is remarked in the first line of its proof
in \cite{BDF1}, the assumption $r\ge 2$ is needed only for the existence of a Stein neighborhood. 
Using \cite[Corollary  1.3]{FFAsian} instead of \cite[Theorem 2.6]{BDF1} we obtain
the same result for all $r\in\Z_+$. In \S\ref{proof} below we shall use
these results with $r=0$.

%
%
%
%
\begin{lemma}
\label{sprays-exist}
{\em (Existence of sprays)} 
Assume that $\ell$, $r$, $k$, $D$, $\sigma$ and $X$ are as in 
Definition \ref{Spray}.  Given a map 
$f_0\colon \bar D\to X$ of class $\cA^r(D)$ to 
a complex manifold $X$, there exists a spray 
$f\colon \bar D\times P\to X$ of class $\cA^r(D)$, 
with the exceptional set $\sigma$ of order $k$, such that 
$f(\cdotp,0)=f_0$.  
\end{lemma}

%
%
%
%
\begin{definition}
\label{Cartan-pair}
Let $\ell\ge 2$ be an integer.
A pair of open subsets $D_0,D_1 \Subset S$ in a Stein manifold $S$
is said to be a {\em Cartan pair} of class $\cC^\ell$ if 
\begin{itemize}
\item[(i)]  $D_0$, $D_1$, $D=D_0\cup D_1$ and $D_{0,1}=D_0\cap D_1$ 
are strongly pseudoconvex with $\cC^\ell$ boundaries, and 
\item[(ii)] $\overline {D_0\backslash D_1} \cap \overline {D_1\backslash D_0}=\emptyset$ 
(the separation property). 
\end{itemize}
\end{definition}

The following is the main result on gluing sprays (see \cite[Proposition 4.3]{BDF1}).
The key ingredient in the proof is a Cartan-type splitting lemma; for a simple
proof see \cite[Lemma 3.2]{FFAsian}.

%
%
%
%
\begin{proposition}
\label{gluing-sprays}
{\em (Gluing sprays)} 
Let $(D_0,D_1)$ be a Cartan pair of class $\cC^\ell$ $(\ell\ge 2)$ in 
a Stein manifold $S$ (Def.\ \ref{Cartan-pair}). 
Set $D=D_0\cup D_1$, $D_{0,1}=D_0\cap D_1$. Let
$X$ be a complex manifold.
Given integers $r\in\{0,1,\ldots, \ell\}$, $k\in\Z_+$, 
and a spray $f\colon \bar D_0\times P_0\to X$ of class $\cA^r(D_0)$ 
with the exceptional set $\sigma$ of order $k$ such that 
$\sigma \cap \bar D_{0,1}=\emptyset$, there is an open 
set $P \ss P_0$ containing $0\in\C^n$ satisfying the following. 

For every spray $f' \colon \bar D_1 \times P_0 \to X$
of class $\cA^r(D_1)$, with the exceptional set $\sigma'$ of order $k$,
such that $f'$ is sufficiently $\cC^r$-close to $f$ on 
$\bar D_{0,1} \times P_0$ and $\sigma'\cap \bar D_{0,1}=\emptyset$, 
there exists a spray $g\colon \bar D \times P \to X$ 
of class $\cA^r(D)$, with the exceptional set $\sigma\cup \sigma'$ 
of order $k$, enjoying the following properties:
\begin{itemize}
\item[(i)] the restriction $g\colon \bar D_0\times P \to X$ 
is close to $f \colon \bar D_0\times P \to X$ in the $\cC^r$-topology
(depending on the $\cC^r$-distance of $f$ and $f'$ on $\bar D_{0,1} \times P_0$),
\item[(ii)] the core map $g_0=g(\cdotp,0)$ is homotopic to 
$f_0=f(\cdotp,0)$ on $\bar D_0$, 
and $g_0$ is homotopic to $f'_0=f'(\cdotp,0)$ on $\bar D_1$, and
\item[(iii)] $g_0$ agrees with $f_0$ up to order $k$ on $\sigma$, 
and it agrees with $f'_0$ up to order $k$ on $\sigma'$.
\end{itemize}
\end{proposition}

\begin{remark}
It follows from the proof in \cite{BDF1} that, in addition to the above, we have
$g(z,t)\in\{f'(z,s)\colon s \in P_0\}$ for each $z\in \bar D_1$ and $t\in P$.
\qed \end{remark}

%
%
%
%
%
%
\section{Proof of Theorem \ref{Main1}}
\label{proof}
The scheme of proof is exactly as in \cite[proof of Theorem 1.1]{BDF1}.
A holomorphic map $f\colon D\to X$ satisfying the conclusion 
of Theorem \ref{Main1} is obtained as a locally uniform limit 
$f=\lim_{j\to \infty} f_j$ in $D$ of a sequence of continuous maps 
$f_j\colon \bar D\to X$  that are holomorphic in $D$. 
At every step of the inductive construction 
we obtain the next map $f_{j+1}$ from $f_j$
by first lifting a  part of the boundary $f_j(bD)$,
lying in a local chart of $X$, to higher levels
of $\rho$, while at the same time  taking care not to drop 
the boundary substantially lower with respect to $\rho$.
The local modification, provided by Lemma \ref{MainLemma},
uses special holomorphic peak functions;
its proof  relies on the work of A.\ Dor \cite{Dor1}.
To pass a critical level of $\rho$ we use methods 
developed in \S\ref{critical} above. 

For technical reasons we work with sprays of maps (see \S\ref{spray}). 
This allows us to patch any local modification, furnished by 
Lemma \ref{MainLemma}, with the given global map $\bar D\to X$  by 
appealing to Proposition \ref{gluing-sprays} above.
When talking of sprays, we adopt the following convention:

%
%
{\em All sprays in this section are assumed to be of class $\cA^0(D)$}, 
and (unless otherwise specified) 
{\em their exceptional set $\sigma$ of order $k$ 
equals the finite set $\sigma$ from Theorem \ref{Main1}.} 
We shall accordingly omit the phrases `of class $\cA^0(D)$' and 
`with exceptional set $\sigma$ of order $k$' 
when there is no ambiguity.

In the lifting process we have to consider two cases: 
The first is to lift the boundary of $D$ across noncritical 
levels of $\rho$, and the second is crossing a critical 
level set of $\rho$. We reduce the second case to the first 
one by using Lemma \ref{crossing}
(see Lemma \ref{bigstep} below).

For maps from strongly pseudoconvex domains to a 
Euclidean space, relevant lifting techniques using holomorphic peak functions
have been developed by several authors. The following result was proved by
A.\ Dor (see \cite[Lemma 1]{Dor1}; here we use Dor's original notation).
The term `normalized-3' indicates that the complex Hessian is 
globally bounded from below with factor $3$, that is, 
$\cL_\rho(x;v)\ge 3|v|^2$. Since we do not wish to
normalize our exhaustion function, we shall need one 
additional constant (denoted $\mu_0$ in Lemma \ref{LemmaDorL} below)
in the corresponding estimates.

%
%
%
%
%
\begin{lemma}
\label{LemmaDor}
{\rm (A.\ Dor, \cite[Lemma 1]{Dor1})}
Assume that $N\ge 2$ and $M\ge N+1$ are integers, $\Omega$ is a domain
in $\C^M$, $\rho\colon\Omega \to\R$ is a smooth normalized-3 
plurisubharmonic function $(\cL_\rho(x;v)\ge 3|v|^2)$, $D\ss \C^N$ is 
a strongly pseudoconvex domain with smooth boundary, 
and $z_0\in bD$. Let $K_1\subset\Omega$ be a compact subset such
that $d\rho\ne 0$ on $K_1$. Then there exist 
\begin{itemize}
\item a constant $\epsilon_0\in (0,1)$ that depends only on $N$ 
and on the domain $D$, 
\item constants $\gamma_0\in (0,1)$ and $C>1$ that depend only on
$K_1$, $\rho$ and $\Omega$, and 
\item
a neighborhood $U\subset bD$ of $z_0$ that depends only on $z_0$ 
and $D$, 
\end{itemize}
such that the following hold.
Given a smooth map $f\colon\bar D \to \Omega$ that is holomorphic on $D$, 
a compact subset $K\subset D$, a number $\e>0$, and a continuous map 
$\gamma\colon bD\to (0,\gamma_0]$, there is a smooth map 
$g\colon\bar D \to \C^M$ that is holomorphic on $D$ and enjoys
the following properties:
\begin{itemize}
\item[(i)] $f(z)+g(z)\in \Omega$ for $z\in \bar D$,
\item[(ii)] $C|g(z)|^2+\e>\rho((f+g)(z))-\rho(f(z))>|g(z)|^2-\e$ for $z\in \bar D$,
\item[(iii)] $|g(z)|>\e_0\gamma(z)$ for $z\in U\cap f^{-1}(K_1)$,
\item[(iv)] $|g(z)|<\e_0^{-1}\gamma(z)$ for $z\in bD$, and
\item[(v)] $|g(z)|<\e$ for $z\in K$.
\end{itemize}
\end{lemma}

Although Dor stated this result 
only for strongly plurisubharmonic functions $\rho$, 
in the final pages of his paper he also proved it for $q$-convex functions with $q\le M-N$; he called such functions `locally $(N+1)$-dimensional plurisubharmonic'.
The proof in \cite{Dor1} is split into three parts. 
In the first part the author chooses a good system of peak functions 
near $z_0$ and obtains a constant $\epsilon_0>0$ that depends only 
on the geometry of the boundary $bD$ near the chosen point $z_0\in bD$, 
but is independent of the target domain $X$ and of the function $\rho$. 
In the second step he constructs a local correction map
that is defined in a neighborhood $U$ of $z_0$ and enjoys  
the stated properties on $U$. In the last step this local map 
is patched with the original global map by solving a $\dibar$-equation 
on $\bar D$.

Since in our case $X$ is a manifold (and not a Euclidean space as in \cite{Dor1}), 
we perform Dor's corrections on small pieces of $\bar D$ near the boundary of $bD$
that are mapped to local charts of $X$, and then glue 
these corrections with the initial spray by the methods
explained in \S\ref{spray}. So we only need the following 
{\em local version of Lemma \ref{LemmaDor}} (before globalization).
We adjust the notation to the one used in the remainder of this section,  
writing $p$ instead of $z_0$, $\e_p$ instead of $\e_0$, and $g$ instead of $f+g$. 
We emphasize that Lemma \ref{LemmaDorL} is what Dor actually proved in 
\cite{Dor1}, and hence it does not require a proof.

\begin{figure}[ht]
\psset{unit=0.5cm, linewidth=0.7pt} 

\begin{pspicture}(-8,-3.5)(8,5.5)

\psarc[linewidth=1pt](0,-8){9}{40}{72}
\psarc[linewidth=1pt](0,-8){9}{97}{140}
\psarc[linestyle=dashed,dash=2pt 1.5pt](0,-8){9}{72}{97}

\psdot[dotsize=3.5pt](0,1)
\rput(0,1.5){$p$}

\rput(5.2,-1.5){$D$}
\rput(1,0){$D_1$}
\rput(-4.8,3){$U_p$}
\rput(-2.7,3){$V_p$}
\rput(1,2.2){$C$}
\psline[linewidth=0.2pt]{<-}(1,1)(1,1.8)

\psecurve[linewidth=1pt](0,-4)(-1.4,0.87)(-1.8,0)(0,-1)(2,-1)(3.5,-0.3)(3,0.47)(1,-1)

\pscircle[linestyle=dotted,linewidth=1pt](0,1){2.8}
\pscircle[linestyle=dotted,linewidth=1pt](0,1){4.5}

\end{pspicture}

\caption{The sets in Lemma \ref{LemmaDorL}}
\label{Lemma5.2}
\end{figure}
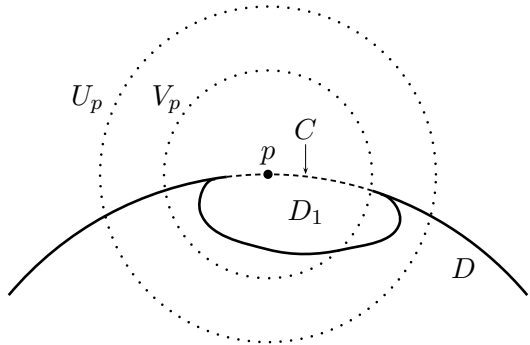

%
%
%
%
%
\begin{lemma}
\label{LemmaDorL}
Let $d\ge 2$, $n\ge d+1$, and $q\le n-d$ be integers. 
Assume that $\omega$ is a domain in $\C^n$,
$\rho\colon\omega \to\R$ is a smooth $q$-convex function, 
$K_\omega$ is a compact subset of $\omega$ such that 
$d\rho\ne 0$ on $K_\omega$, $D\ss \C^d$ is a domain with smooth boundary, 
and $p\in bD$ is a strongly pseudoconvex boundary point of $D$.  
Then there exist 
\begin{itemize}
\item[(a)] small balls $V_p\Subset U_p$ in $\C^d$ 
such that $p\in V_p$, 
\item[(b)] a constant $\epsilon_p\in (0,1)$ 
depending only on $U_p\cap bD$, 
\item[(c)]  a constant $\mu_0>0$ depending only on $\rho$ and $K_\omega$, and 
\item[(d)] a constant $\gamma_0\in (0,1)$ 
depending on $U_p\cap bD$, $\rho$ and $K_\omega$,
\end{itemize}
such that the following hold.
Given an open subset $D_1\subset U_p\cap D$, 
an open subset $C$ of $bD$ contained in $V_p$ such that 
${\rm dist}_{\C^d}(\overline C, bD_1 \bs bD) >0$,
a smooth map $f\colon \bar D_1 \to \omega$ that is holomorphic on $D_1$,
a number $\e>0$, and a continuous map $\gamma\colon bD\cap bD_1\to (0,\gamma_0]$,
there is a smooth map $g\colon \bar D_1 \to \omega$ that is
holomorphic on $D_1$ and enjoys the following properties:
\begin{itemize}
\item[(i)] $\rho(g(z))-\rho(f(z)) >\mu_0|g(z)-f(z)|^2-\e$ for $z\in \bar D_1$,
\item[(ii)] $|g(z)-f(z)|>\e_p\gamma(z)$ for $z\in C\cap f^{-1}(K_\omega)$,
\item[(iii)] $|g(z)-f(z)|<\e_p^{-1}\gamma(z)$ for $z\in bD\cap bD_1$, and
\item[(iv)] $|g(z)-f(z)|<\e$ for points $z\in \bar D_1$ such that 
${\rm dist}_{\C^d}(z,C)>\e$.
\end{itemize}
\end{lemma}

The main sets in Lemma \ref{LemmaDorL} are illustrated on Figure \ref{Lemma5.2},
with $C$ shown as the dashed arc on $bD \cap bD_1$.

Using Lemma \ref{LemmaDorL} we now prove 
our main modification lemma for the noncritical case.
Note that $f_0$ always denotes the core map of a spray $f$.

%
%
%
%
\begin{lemma}
\label{MainLemma}
Let $d\ge 2$, $n\ge d+1$ and $k\ge 0$ be integers.         
Assume that $X$ is an $n$-dimensional complex manifold endowed
with a complete metric ${\rm dist}$, $\Omega$ is an open subset of $X$,
and $\rho\colon\Omega\to \R$ is a smooth function 
such that for a pair of real numbers 
$c_1<c_2$ the set 
\[
	\Omega_{c_1,c_2}= \{x\in\Omega\colon c_1\le \rho(x)\le c_2\}
\]
is compact, $d\rho\ne 0$ on $\Omega_{c_1,c_2}$, and  the Levi form 
$\cL_\rho$ of $\rho$ has at least $d+1$ positive eigenvalues at every point 
of $\Omega_{c_1,c_2}$.

Let $D\Subset S$ be a smoothly bounded, strongly pseudoconvex domain in a 
Stein manifold $S$ of dimension $d$ 
and let $\sigma$ be a finite set of points in $D$.
Choose real numbers $c'_1,c'_2$ such that $c_1<c'_1<c'_2< c_2$.
Then there is a number $\delta>0$ with the following property.
Given a number $c\in[c'_1,c'_2]$, a compact set $K_D\subset D$, and a spray of maps 
$f\colon\bar D\times P\to X$ with the exceptional set $\sigma$ of order $k$
such that  
\[
	f_0(z)\in\Omega\ (\forall z\in\overline{D\bs K_D}),\quad  
  \rho(f_0(z))>c-\delta\  (\forall z\in bD), 
\]
there exist for each $\e>0$ 
an open set $P_0 \ss P$ containing the origin and a spray 
$g\colon\bar D\times P_0\to X$ with the exceptional set 
$\sigma$ of order $k$ such that
\begin{itemize}
\item[(i)]  $g_0(z)\in \Omega$ and  $\rho(g_0(z))>c+\d$ for $z\in bD$,
\item[(ii)] $g_0(z)\in \Omega$ and  
$\rho(g_0(z))>\rho(f_0(z))-\e$ for $z\in \overline{D\bs K_D}$,
\item[(iii)] ${\rm dist }(f_0(z),g_0(z))<\e$ for $z\in K_D$, and
\item[(iv)] the maps $f_0$ and $g_0$ have the same $k$-jets at every point in $\sigma$,
and $g_0$ is homotopic to $f_0$ relative to $\sigma$.
\end{itemize}
\end{lemma}

\begin{proof}
We first explain the main idea.
Lemma \ref{LemmaDorL} provides a local step on both sides --- 
locally with respect to the boundary $bD$,
and locally on the level set of $\rho$ in $X$. 
In every step of the inductive construction  we lift the part of 
the image of the boundary $bD$ that lies in a small 
coordinate neighborhood to higher levels of the exhaustion function 
$\rho$ (see Sublemma \ref{sublemma}).
In finitely many such steps we push the image 
of the boundary outside a certain bigger sublevel set of $\rho$.
Each step in the construction consists of finitely many substeps,
and at each substep we only make corrections on the part of 
the boundary lying in a suitable coordinate neighborhood in $S$. 
Sprays are used at every substep to patch the local correction 
with the previous global map. 

Now to the details. Let $\B^n$ denote the open unit ball in $\C^n$, 
and let $s\B^n$ denote the ball of radius $s >0$. 
Fix a number $c\in [c'_1,c'_2]$. We shall find a number $\delta>0$ 
satisfying the conclusion of Lemma \ref{MainLemma} for this value of $c$. 
It will be clear from the construction that $\delta$ can be 
chosen uniformly for all $c'$ sufficiently close to $c$, 
and hence (by compactness) for all $c\in[c'_1,c'_2]$.

Since the level set $\{\rho=c\}$ is compact, 
there are finitely many holomorphic coordinate maps 
$h_i\colon\tfrac54\B^n\to X$ $(i=1,\ldots,N)$ such that
\[
	\{\rho=c\}\subset \bigcup_{i=1}^{N} h_i(\tfrac14\B^n) \subset 
	\bigcup_{i=1}^{N} h_i(\tfrac54\B^n) \subset \Omega_{c_1,c_2}.
\]
For each point $p\in bD$ one can choose local holomorphic 
coordinates in $S$, and in this coordinate patch we obtain the sets 
$U_p$, $V_p$ and a constant $\e_p$ as in Lemma \ref{LemmaDorL} (parts (a) and (b)).
Choose open coverings $\{V_j\}_{j=1}^M$ and $\{U_j\}_{j=1}^M$
of $bD$ such that each pair $V_j\Subset U_j$ 
corresponds to $V_{p_j} \Subset U_{p_j}$ for some point $p_j\in bD$,
and $U_j\cap \sigma=\emptyset$ for each $j$. We also obtain the 
corresponding numbers $\e_j=\e_{p_j}>0$. 

Let $\e_0=\min\{\e_1,\ldots,\e_M\}>0$. 
Using Lemma \ref{LemmaDorL} (parts (c) and (d)) 
for the data $\bar D\cap U_j$ (in the local coordinates), 
$\omega=\tfrac54\B^n$, $K_\omega=\overline \B^n$, and with $\rho$ 
replaced  by the function $\rho\circ h_i$, we also obtain constants 
$\gamma_i^j$ and $\mu_i^j$ for $i=1,\ldots,N$ and $j=1,\ldots,M$ 
(these correspond to $\gamma_0$, resp.\ to $\mu_0$, 
in Lemma \ref{LemmaDorL}). 
Choose constants $\alpha>0$ and $\beta>0$ such that 
the following hold for  $i=1,\ldots,N$:
\begin{align}
	w,w'\in\overline \B^n \Longrightarrow {\rm dist}(h_i(w),h_i(w')) 
			&\le \alpha|w-w'|,\label{defalpha}\\
	\bigl( x\in h_i(\tfrac78\B^n),\ x'\in X,\ {\rm dist}(x,x')<M\beta \bigr)
	&  \Longrightarrow \label{defLambda}\\
	\bigl( x'\in h_i(\B^n),\ |h_i^{-1}(x)&-h_i^{-1}(x')| < \tfrac1{8N} \bigr). \nonumber
\end{align}
Set 
\begin{equation}
\label{gamma}
	\gamma=\min\{\gamma_i^j,\tfrac{\epsilon_0\beta}{3\,\alpha}\}>0, 
	\qquad \mu=\min\{\mu_i^j\}>0,
\end{equation}	
the minima being taken over all indices $i=1,\ldots,N$  and  $j=1,\ldots,M$.
(This choice of $\gamma$ insures that our correction is  small 
compared to the size of $\omega$.)
Finally, we choose a number $\delta$ such that  
\begin{equation}
\label{delta}
	0<\delta<\tfrac13 \mu\e_0^2\gamma^2,  \qquad                          
	\Omega_{c-\delta,c+2\d} \subset 
	\bigcup_{i=1}^{N}h_i\left(\tfrac14\B^n\right).                       
\end{equation}
We shall prove that this $\delta$ satisfies Lemma \ref{MainLemma}.
We need the following

%
%
%
%
%
\begin{sublemma} 
\label{sublemma}
Fix an index $i\in\{1,2,\ldots, N\}$.
Given a spray of maps $f'\colon\bar D\times P\to X$,  
with the exceptional set $\sigma$ of order $k$ and such that
$f_0'(z)\in\Omega$ for $z\in \overline{D\bs K_D}$, there exist
for each $\e'>0$ an open set $P' \ss P$ containing the origin 
and a spray of maps $g'\colon\bar D\times P'\to X$ 
with the exceptional set $\sigma$ of order $k$ such that
$g'_0(z)\in \Omega$ for $z\in \overline{D\bs K_D}$ and
the following hold:
\begin{itemize}
\item[(i')] $\rho(g'_0(z))>\rho(f'_0(z))+ 3\d-\e'$   
for $z\in bD$ such that $f'_0(z)\in h_i(\tfrac 12\B^n)$,
\item[(ii')] $\rho(g'_0(z))>\rho(f'_0(z))-\e'$ for $z\in \overline{D\bs K_D}$,
\item[(iii')] ${\rm dist }(f'_0(z),g'_0(z))<M\beta$ for $z\in \bar D$, 
\item[(iv')] ${\rm dist }(f'_0(z),g'_0(z))<\e'$ for $z\in K_D$, and
\item[(v')] the maps $f'_0$ and $g'_0$ have the same $k$-jets at every point in $\sigma$.
\end{itemize}
\end{sublemma}

\begin{proof}
Let $f^0=f'$ and $P^0=P$.  Recall that $bD\subset \bigcup_{j=1}^M V_j$.
We inductively construct a finite decreasing 
sequence of parameter sets $P^0\supset P^1\supset \cdots\supset P^M$,
with $0\in P^{j+1}\ss P^{j}$ for $j=0,\ldots,M-1$, and a sequence of
sprays $f^j\colon\bar D\times P^j\to X$ with the exceptional set $\sigma$ 
of order $k$ such that the following hold for $j=0,1,\ldots,M-1$:
\begin{itemize}
\item[(i$^\dagger$)] $\rho(f^{j+1}_0(z))>\rho(f^j_0(z))+3\d-\tfrac{\e'}M$        
for every point $z\in bD\cap V_{j+1}$ such that $f^j_0(z)\in h_i(\tfrac 34\B^n)$,
\item[(ii$^\dagger$)] $\rho(f^{j+1}_0(z))>\rho(f^{j}_0(z))-\tfrac{\e'}M$ for $z\in \overline{D\bs K_D}$,
\item[(iii$^\dagger$)] ${\rm dist }\bigl(f^{j+1}_0(z),f^{j}_0(z)\bigr) <{\beta}$ for $z\in \bar D$,
\item[(iv$^\dagger$)] ${\rm dist }\bigl(f^{j+1}_0(z),f^{j}_0(z)\bigr) < \tfrac{\e'}M$ for $z\in K_D$, and
\item[(v$^\dagger$)] the maps $f^{j+1}_0$ and $f^{j}_0$ 
have the same $k$-jets at every point in $\sigma$.
\end{itemize}

Assume for a moment that we have already constructed the sequences $P^j$ and $f^j$.
Let $P'=P^M$ and $g'=f^M$. Using (ii$^\dagger$)--(v$^\dagger$) repeatedly $M$ times
we see that properties (ii')--(v') in Sublemma \ref{sublemma} hold. 
To see that  (i') holds, fix 
a point $z\in bD$ such that $f^0_0(z)\in h_i(\tfrac12 \B^n)$. 
Choose an index $j$ such that $z\in V_j$. 
By (iii$^\dagger$) and (\ref{defLambda})
it follows that $f^{j-1}_0(z)\in h_i(\tfrac34 \B^n)$,
and thus (i$^\dagger$) gives $\rho(f^{j}_0(z))>\rho(f^{j-1}_0(z))+ 3\d-\tfrac{\e'}M$.
Using (ii$^\dagger$) repeatedly this implies  
$\rho(g'_0(z))>\rho(f'_0(z))+3\d-{\e'}$. 
Therefore $g'$ enjoys all required properties. 

It remains to construct the sequences $P^j$ and $f^j$.
Assume inductively that we have already constructed 
$P^0,\ldots,P^j$ and $f^0,\ldots,f^j$ 
for some $j\in\{0,1,\ldots,M-1\}$; we now explain how
to find $P^{j+1}$ and $f^{j+1}$. Set 
\[
	C=  bD\cap V_{j+1}\cap(f_0^j)^{-1}(h_i(\tfrac{13}{16}\B^n)).
\]
Observe that the open set
$
	D\cap V_{j+1}\cap \bigl(f_0^j\bigr)^{-1}\left(h_i\left(\tfrac{13}{16}\B^n\right)\right)
$
is pseudoconvex, contained in 
$
	U_{j+1}\cap \bigl(f_0^j\bigr)^{-1}\left(h_i\left(\tfrac{7}{8}\B^n\right)\right)
	\subset U_{j+1}\cap \bar D,
$
and has positive distance to 
$
	\bar D\bs \left( U_{j+1}\cap  \bigl(  f_0^j\bigr)^{-1}\left(h_i\left(\tfrac{7}{8}\B^n\right)\right) \right). 
$ 
Hence there is a smoothly bounded, strongly pseudoconvex
domain $D_1$ contained in $D$ such that 
\[
	V_{j+1}\cap \bigl(f_0^j\bigr)^{-1} \left(h_i\left(\tfrac{13}{16}\B^n\right)\right)
  \subset \bar D_1\subset 
  U_{j+1}\cap \bigl(f_0^j\bigr)^{-1}\left(h_i\left(\tfrac{7}{8}\B^n\right)\right)
\]
and 
\[
	{\rm dist}_{\C^d}(\overline C, bD_1 \bs bD) >0.
\]
The situation is as shown in Figure \ref{Lemma5.2}, with $U_p$ replaced by $U_{j+1}$
and $V_p$ replaced by $V_{j+1}$. 

Choose a smoothly bounded, strongly pseudoconvex domain $D_0\subset D$, 
obtained by denting $D$ slightly inward in a neighborhood of $C$, such that
$D\bs D_0\subset U_{j+1}$, $bD_0\cap \overline C=\emptyset$, and 
$(D_0,D_1)$ is a Cartan pair such that $D_0\cup D_1=D$ 
(see Definition \ref{Cartan-pair}). Set
\[
		A_{i,j}(z,t)= h_i^{-1}\circ f^j(z,t) - h_i^{-1}\circ f^j(z,0).
\]
There exists a smaller parameter set $0\in P^j_0\subset P^j$ such that for 
$w\in \tfrac{15}{16}\B^n$, $z\in \bar D_1$ and $t\in P^j_0$ the following hold:
\begin{align}
	|A_{i,j}(z,t)| &< \tfrac1{16}, 
\label{defPl1}\\
	|\rho(h_i(w))- \rho\bigl( h_i(w+ A_{i,j}(z,t)) \bigr)| &< \tfrac{\e'}{3M},
\label{defPl2}\\
	{\rm dist}\bigl( h_i(w), h_i(w+A_{i,j}(z,t))\bigr) &< \tfrac\beta {3}.
\label{defPl3}
\end{align}
Applying Lemma \ref{LemmaDorL} to the map 
$h_i^{-1}\circ f_0^j$ on $\bar D_1$, the constant function $\gamma(z)=\gamma$
(with $\gamma$ as in (\ref{gamma})), 
and a number $\e>0$ (to be specified later), we obtain a map 
$g\colon\bar D_1 \to \C^n$ enjoying the following properties:
\begin{align}
	& \rho(h_i(g(z))) > \rho(f_0^j(z))+ \mu|h_i^{-1}\circ f_0^j(z) -g(z)|^2-\epsilon 
	\ \text{ for }z\in\bar D_1, 
\label{lastng1} \\
	& |h_i^{-1}\circ f_0^j(z)-g(z)| > \epsilon_0\gamma  \ \text{ for }z\in C,
\label{lastng2}\\
	& |h_i^{-1}\circ f_0^j(z) -g(z)| < \epsilon_0^{-1}\gamma 
				\ \text{ for }z\in bD_1\cap bD \text{, and}\label{lastng4}\\
	& |h_i^{-1}\circ f_0^j(z) -g(z)| < \epsilon \ \text{ for }z\in \bar D_1
	\text{ such that }{\rm dist}_{\C^d}(z,C)>\epsilon.
\label{lastng3}
\end{align}

If $\e< \min\{\frac{\gamma}{\epsilon_0}, {\rm dist}_{\C^d}(C,bD_1 \bs bD)\}$ 
then (\ref{lastng4}), (\ref{lastng3}) and the maximum principle imply that 
$|h_i^{-1}\circ f_0^j(z)-g(z)| < \frac{\gamma}{\epsilon_0}$ for $z\in \bar D_1$.
By (\ref{defalpha}) we get
\[
	{\rm dist}\bigl(h_i(g(z)),f_0^j(z)\bigr) < \tfrac\beta{3}\ \text{ for } z\in \bar D_1.
\]
By (\ref{defLambda}) and (\ref{defPl3}) this allows us to define a spray
$f'\colon\bar D_1\times P^j_0 \to X$ (with empty exceptional set) 
by setting
\[
		f'(z,t)=h_i\bigl(g(z)+A_{i,j}(z,t) \bigr)
			\ \text{ for }z \in \bar D_1,\,t\in P^j_0.
\]
If $\epsilon<\tfrac{\epsilon'}{3M}$ then it follows from (\ref{defPl2}),
(\ref{lastng1}), (\ref{lastng2})  and the definition
of $\delta$ that
\begin{align}
	\rho(f'(z,t))&> \rho(f_0^j(z))+3\delta-\tfrac{2\e'}{3M}\ \text{ for }z\in C, \label{lastn1}\\    
	\rho(f'(z,t))&>\rho(f_0^j(z))-\tfrac{2\e'}{3M} \ \text{ for }z\in\bar D_1.\label{lastn2}
\end{align}
If $\epsilon>0$ is small enough then for every $z\in \overline{D_0\cap  D}_1$ 
we have ${\rm dist}_{\C^d}(z,C)>\epsilon$, and the properties  
(\ref{lastng3}) and (\ref{defalpha}) imply that
\[
	{\rm dist}\bigl(f'(z,t),f^j(z,t)\bigr) \le \alpha\epsilon  \ \text{ for }
		(z,t)\in \overline{D_0\cap  D}_1 \times P_0^j.
\]
Finally, if $\e>0$ is small enough then  we can
glue the sprays $f^j$ and $f'$ by Proposition \ref{gluing-sprays}. 
This gives a smaller parameter set $0\in P^{j+1} \subset P^j_0$ and
a new spray $f^{j+1}\colon \bar D \times P^{j+1} \to X$ 
whose restriction $f^{j+1}\colon \bar D_0\times P^{j+1} \to X$ 
is as close as desired to the spray 
$f^j \colon \bar D_0\times P^{j+1} \to X$ in the $\cC^0$-topology,  
and the range of $f^{j+1}$ over $\overline D_1$ is contained 
in the range of the spray $f'$. 
The good approximation of $f^{j}$ by $f^{j+1}$ over $D_0$ and properties
(\ref{lastn1}) and (\ref{lastn2}) ensure that properties 
(i$^\dagger$)--(v$^\dagger$) hold. 
\end{proof}
%
%
%
%
%

We now conclude the proof of Lemma \ref{MainLemma}.
Let $\e'=\min\{\tfrac\e N,\tfrac\d N\}$, $f^0=f$ and $P^0=P$. 
We construct a decreasing sequence of 
open parameter sets $P^0\supset P^1\supset\cdots\supset P^N$,
with $0\in P^j\ss P^{j-1}$ for $j=1,\ldots,N$, 
and a sequence of sprays $f^j\colon\bar D\times P^j\to X$ 
with the exceptional set $\sigma$ of order $k$
such that the following hold for $j=1,\ldots, N$:
\begin{itemize}
\item[(i'')] $\rho(f^{j}_0(z))>\rho(f^{j-1}_0(z))+3\d-\e'$   
when $z\in bD$ and $f^{j-1}_0(z)\in h_j(\tfrac 12\B^n)$, 
\item[(ii'')] $\rho(f^{j}_0(z))>\rho(f^{j-1}_0(z))-\e'$ for $z\in \overline{D\bs K_D}$,
\item[(iii'')] if $z\in bD$ and $f_0(z)\in h_i\bigl(\tfrac 14\B^n\bigr)$ 
for some $i\in\{1,2,\ldots, N\}$ 
then $f^{j}_0(z)\in h_i\bigl((\tfrac 14+\tfrac{j}{4N})\B^n\bigr)$,
\item[(iv'')] ${\rm dist }\bigl(f^{j}_0(z),f^{j-1}_0(z)\bigr)<\tfrac\e N$ for $z\in K_D$, and
\item[(v'')] the maps $f^{j}_0$ and $f^{j-1}_0$ 
have the same $k$-jets at every point in $\sigma$.
\end{itemize}

Assume inductively that we have already constructed 
$f^0,\ldots, f^{j-1}$ and $P^0,\ldots,P^{j-1}$ for some 
$j\in\{1,\ldots,N\}$. 
We use Sublemma \ref{sublemma} for $f'=f^{j-1}$ to obtain 
the next spray $f^{j}=g' \colon \bar D\times P^j \to X$.
Properties (i'), (ii'), (iv') and (v') in the Sublemma imply 
the corresponding properties (i''), (ii''), (iv'') and (v'') above. 
Property (iii'') follows from (iii') and (\ref{defLambda}). 
This completes the induction step and hence gives the
desired sequences. 

We now show  that Lemma \ref{MainLemma} holds for  
the parameter set $P_0=P^N$ and 
the spray $g=f^N \colon \bar D\times P^N\to X$. 
The properties (ii)--(iv) follow easily from 
the inductive construction above. To prove (i), choose a point $z\in bD$. 
By (\ref{delta}) we either have $\rho(f_0(z))>c+2\d$ 
(and in this case the property (ii'') implies that $\rho(g_0(z))>c+\d$),
or else $f_0(z)\in h_i(\tfrac14\B^n)$ for some $i\in\{1,\ldots,N\}$.
In the latter case we get by (iii'') that 
$f^{i-1}_0(z)\in  h_i(\tfrac12\B^n)$, and therefore property (i'') implies 
$\rho(f^{i}_0(z))> \rho(f^{i-1}_0(z))+ 3\d-\e'$.        
Using this together with (ii'') and $\rho(f_0(z))>c-\delta$ we obtain
\[
 \rho(g_0(z))\ge \rho(f_0(z))+3\d-N\e' \ge \rho(f_0(z))+2 \d>c+\d.
\]
This proves Lemma \ref{MainLemma}.  
\end{proof}

Using Lemma \ref{MainLemma} we now prove the following result 
that provides the lifting construction in the proof of Theorem \ref{Main1}.

%
%
%
%
\begin{lemma}
\label{bigstep}
Let $X$, $K$, $\Omega$, $\rho$, $r$, and $D\Subset S$ be as in Theorem \ref{Main1}. 
Choose a complete metric ${\rm dist}$ on X inducing the manifold topology.
Let $P$ be an open set in $\C^m$ containing the origin, and let $0<M_1<M_2$. 
Assume that $f\colon\bar D\times P\to X$ is a spray of maps 
with the exceptional set $\sigma$ of order $k$ 
and $U\ss D$ is an open subset such that 
$f_0(\bar D\bs U)\subset \{x\in \Omega\colon\rho(x)>M_1\}$.
Given $\e>0$, there exist a domain $P'\subset P$
containing $0\in\C^m$ and a spray of maps $g\colon\bar D\times P'\to X$ 
with the exceptional set $\sigma$ of order $k$ enjoying the following:
\begin{itemize}
\item[(i)]   $g_0(z)\in \{x\in \Omega\colon\rho(x)>M_2\}$ for $ z\in bD$,
\item[(ii)]  $g_0(z)\in \{x\in \Omega\colon\rho(x)>M_1\}$ for $z\in\bar D\bs U$,
\item[(iii)] ${\rm dist }(g_0(z),f_0(z))<\e$ for $z\in \overline U$, 
\item[(iv)]  $f_0$ and $g_0$ have the same $k$-jets at each of the points in $\sigma$, and 
\item[(v)]   $g_0$ is homotopic to $f_0$ relative to $\sigma$.
\end{itemize}
\end{lemma}

\begin{proof}
After a small change of $M_1$ and $M_2$ we may assume that these 
are regular values of $\rho$. By a finite  subdivision of 
$[M_1,M_2]$ it suffices to consider the following two cases:

{\em Case 1}:  $\rho$ has no critical values on $[M_1,M_2]$.
In this {\em noncritical case}  we obtain $g$ by applying Lemma \ref{MainLemma} 
finitely many times. 

{\em Case 2:} $\rho$ has exactly one critical point $p$
in $\{x\in \Omega\colon M_1\le \rho(x)\le M_2\}$
(the {\em critical case}).

In Case 2 we follow \cite[proof of Theorem 1.1, \S6]{BDF1}.
We have $M_1< \rho(p)<M_2$. Choose $c_0>0$ so small that 
$M_1+3c_0 <\rho(p)< M_2-3c_0$ and
$f_0(z)\in \{x\in \Omega\colon\rho(x)> M_1+3c_0\}$ for all $z\in\bar D\bs U$.
Set 
\[
	K_{c_0}=\{x\in \Omega\colon \rho(p)-c_0\le \rho(x) \le \rho(p)+3c_0\}. 
\]
Lemma \ref{crossing} furnishes a constant $t_0\in (0,c_0)$, 
a smooth function 
\[
	\tau \colon \{x\in\Omega\colon \rho(x) \le \rho(p)+3c_0\} \to\R,
\]
and an embedded disc  $E\subset \Omega$ of dimension equal to 
the Morse index of $\rho$ at $p$, enjoying the following:
\begin{itemize}
\item[(a)]   $\{\rho\le \rho(p)-c_0\} \cup E \subset \{\tau\le 0\} \subset \{\rho\le \rho(p)-t_0\}\cup E$, 
\item[(b)]  $\{\rho \le \rho(p)+c_0\} \subset \{\tau \le 2c_0\} \subset 
\{\rho<\rho(p)+ 3c_0\}$,
\item[(c)] $\tau$ is $q$-convex at every point of $K_{c_0}$, and
\item[(d)]  $\tau$  has no critical values in $(0,3c_0) \subset \R$.
\end{itemize}

Applying Lemma \ref{MainLemma} finitely many times we get a spray 
$\wt f\colon\bar D\times \wt P\to X$ with exceptional set 
$\sigma$ of order $k$ having the following properties:
\begin{itemize}
\item[(i')]   $\wt f_0(z)\in \{x\in \Omega\colon\rho(x)>\rho(p)-t_0\}$ for $ z\in bD$,
\item[(ii')]  $\wt f_0(z)\in \{x\in \Omega \colon\rho(x)>M_1+2c_0\}$ for $z\in\bar D\bs U$,
\item[(iii')] ${\rm dist }(\wt f_0(z),f_0(z))<\tfrac\e3$ for $z\in \overline U$, and
\item[(iv')]  $f_0$ and $\wt f_0$ have the same $k$-jets at each of the points in $\sigma$.
\end{itemize}

For the parameter values $t\in \wt P$ sufficiently close to $t=0$
we also have $\wt f_t(bD) \subset \{x\in \Omega\colon\rho(x)>\rho(p)-t_0\}$
by (i'). Since $\dim_{\R} E\le 2n-2d$ and $\dim_{\R} bD=2d-1$, 
Sard's lemma gives a $t' \in \wt P$ 
arbitrarily close to the origin such that $\wt f_{t'} (bD) \cap E= \emptyset$. 
By a translation in the $t$-variable we can choose 
$\wt f_{t'}$ as the new central map; the new spray 
(still denoted $\wt f$) then enjoys the following properties
for all $t$ sufficiently near $0$:
\begin{itemize}
\item[(i'')]   $\wt f_t(z)\in \{x\in \Omega\colon\rho(x)>\rho(p)-t_0\}\bs E$ for $ z\in bD$,
\item[(ii'')]  $\wt f_t(z)\in \{x\in X\colon\rho(x)>M_1+c_0\}$ for $z\in\bar D\bs U$,
\item[(iii'')] ${\rm dist }(\wt f_t(z),f_0(z))<\tfrac{2\e}3$ for $z\in \overline U$, and
\item[(iv'')]  $f_0$ and $\wt f_t$ have the same $k$-jets at each of the points in $\sigma$.
\end{itemize}

Since $\{\tau\le 0\} \subset \{\rho\le \rho(p)-t_0\}\cup E$
by property (a), (i'') ensures that $\tau >0$ on $\wt f_t(bD)$.
Since $\tau$ has no critical values on $(0,3c_0)$ by property (d),
we can use the noncritical case (Case 1 above), with $\tau$ 
instead of $\rho$, to push the boundary of 
the central map into the set $\{\tau > 2c_0\}$. 
As $\{\rho \le \rho(p)+c_0\} \subset \{\tau \le 2c_0\}$ by property (b),
the image of $bD$ now lies in $\{\rho > \rho(p) +c_0\}$.
We have thus crossed the critical level $\{\rho=\rho(p)\}$ 
and may continue with the noncritical case procedure, applied 
again with the function $\rho$. In a finite number of steps 
we obtain a spray $g$ with the required properties. 
%
\end{proof}

%
%
%
\begin{proof}[Proof of Theorem \ref{Main1}]
We follow \cite[proof of Theorem 1.1]{BDF1}, 
using Lemma \ref{bigstep} instead of 
\cite[Proposition 6.3]{BDF1}. We begin by embedding the initial
map $f_0\colon\bar D\to X$ into a spray of maps 
$f=f^0\colon \bar D\times P\to X$
(Definition \ref{Spray} and Lemma \ref{sprays-exist})
such that $f(\cdotp,0)=f_0$ and $f(bD\times P)\subset \Omega$.
By inductively applying Lemma \ref{bigstep} we obtain a sequence 
of sprays $f^j\colon\bar D\times P_j\to X$ $(j=1,2,\ldots)$ with decreasing
parameter sets $\C^m\supset P=P_0\supset P_1\supset P_2\supset\cdots$  
containing the origin $0\in \C^m$ such that the maps 
$f^j_0=f^j(\cdotp,0) \colon \bar D\to X$ converge uniformly 
on compacts in $D$ to a holomorphic map 
$f\colon D\to X$ satisfying the conclusion of Theorem \ref{Main1}. 
\end{proof}

%
%
%
%
%
%
\section{Positivity and convexity of holomorphic vector bundles}
\label{vector-bundles}
In this section we recall the notions of {\em positivity} and 
{\em signature} of a Hermitian holomorphic vector bundle
(Griffiths \cite{Griffiths66,Griffiths69}) and its connection with 
the Levi convexity properties of the squared norm function.

Let $\pi\colon E\to M$ be a holomorphic Hermitian vector bundle 
with fiber $\C^r$ over a complex manifold $M$ of dimension $m$.
We identify $M$ with the zero section of $E$.
The metric on $E$ is given in a local frame $(e_1,\cdots,e_r)$
by a Hermitian matrix function $h=(h_{\rho\sigma})$ with 
\[
		h_{\rho\sigma}(x)= \langle e_\sigma(x),e_\rho(x) \rangle, \quad \rho,\sigma=1,\ldots,r.
\]
The Chern connection matrix $\theta$ and the Chern curvature form $\Theta$
are given in any local holomorphic frame by 
\[
	\theta= h^{-1}\di h, \qquad 
	\Theta=\dibar \,\theta = -h^{-1}\di\dibar h + h^{-1}\di h \wedge h^{-1}\dibar h. 
\]
(See \cite[Chapter 5]{Dem-book} or \cite[Chapter III]{Wells}.)
For a line bundle $(r=1)$ with the metric $h=e^{-\psi}$ the above equal 
\[
	\theta= h^{-1}\di h = \di\log h = -\di\,\psi, 
	\quad  \Theta= - \di\dibar \log h = -\dibar\di\,\psi=\di\dibar\,\psi.
\]
In local holomorphic coordinates $z=(z^1,\ldots,z^m)$ on $M$ we have
\[
	\Theta= \sum_{\nad{\rho,\sigma=1,\ldots,r}{i,j=1,\ldots,m}} 
	      \Theta^\rho_{\sigma i j} e^*_\sigma \otimes e_\rho\cdotp dz^i\wedge d\bar z^j.
\]

For any point $x_0\in M$ there exists a local holomorphic frame 
$(e_1,\ldots,e_r)$ which is {\em special at} $x_0$, in the sense
that the associated matrix $h$ satisfies $h(x_0)=I$ and $dh(x_0)=0$
(see \cite[p.\ 195]{Griffiths69}). In this case we get
\begin{equation}
\label{eq:frame}
	\theta(x_0)=0,\quad \Theta(x_0)= -\di\dibar h (x_0),\quad 
	\overline {\Theta^\sigma_{\rho i j} (x_0)} = \Theta^\rho_{\sigma j i}(x_0) =
	- \frac{\di^2 h_{\rho \sigma}}{\di z^i \di \bar z^j} (x_0). 
\end{equation}
To each vector $e=\sum_{\rho=1}^r \xi^\rho e_\rho(x_0) \in E_{x_0}$
we associate the $(1,1)$-covector
\[
	\Theta\{e\}= \frac{\mathrm{i}}{2}\,  \langle \Theta \, e,e\rangle =
	\frac{\mathrm{i}}{2}
	\sum_{\nad{\rho,\sigma=1,\ldots,r} {i,j=1,\ldots,m}} 
	      \Theta^\rho_{\sigma i j}(x_0) \xi^\sigma \bar \xi^\rho \, dz^i\wedge d\bar z^j.
\]
Its coefficients 
$
	A_{i j}(x_0,\xi) = \sum_{\rho,\sigma=1,\ldots,r} 
	      \Theta^\rho_{\sigma i j}(x_0) \xi^\sigma \bar \xi^\rho 
$
form a Hermitian matrix. Denote by $s(e)$ (resp.\ $t(e)$) 
the number of positive (resp.\ negative)
eigenvalues of $\Theta\{e\}$; that is, $(s(e),t(e))$ is the 
signature of the Hermitian quadratic form 
\begin{equation}
\label{curvature-form}
	\C^m\ni \eta \to \sum_{i,j} A_{i j}(x_0,\xi) \eta^i \bar \eta^j
	= \sum_{\nad{\rho,\sigma=1,\ldots,r}{i,j=1,\ldots,m}} 
	      \Theta^\rho_{\sigma i j}(x_0) \xi^\sigma \bar \xi^\rho \eta^i\bar \eta^j.  
\end{equation}
The numbers $s(e), t(e)$ only depend on the Hermitian metric on $E$, 
and not on the particular choices of coordinates and frames. 
The following notions are due to 
Griffiths \cite{Griffiths65,Griffiths66,Griffiths69};
see also \cite{AG} and \cite[Chapter 7]{Dem-book}.

%
%
%
%
\begin{definition}
\label{Griffiths-positive}
The pair of numbers $(s(e),t(e))$ defined above 
is the {\em signature} of the Hermitian 
holomorphic vector bundle $E\to M$ at the point $e\in E\bs M$. 
The signature of $E$ is $(s,t)$ where 
\[
	s=\min\{s(e)\colon e\in E\bs M\},\quad t=\min\{t(e)\colon e\in E\bs M\}.
\]	
$E$ is of {\em pure signature} $(s,t)$ is $s=s(e)$ and $t=t(e)$
for all $e\in E\bs M$; it is {\em positive} (resp.\ {\em negative}) 
if it has pure signature $(m,0)$ (resp.\ $(0,m)$).
\end{definition}

Let $\phi\colon E\to \R_+$ denote the function $\phi(e)= ||e||^2$.
For $c\in (0,\infty)$ set
\[
	W_c =\{e\in E \colon  \phi(e)<c\}, \quad \Sigma_c=bW_c=\{e\in E\colon \phi(e)=c\}.
\]
The following explains the connection 
between the curvature properties of a  Hermitian metric 
and the Levi convexity properties of $\phi$.
(See Andreotti and Grau\-ert \cite[\S 23]{AG} 
and Griffiths \cite[p.\ 426]{Griffiths66}.)

\begin{proposition}
\label{signature}
Let $\pi\colon E\to M$ be a Hermitian holomorphic vector bundle with fiber $\C^r$
over an  $m$-dimensional complex manifold $M$. Set $n=m+r=\dim E$.
Then the following hold:

(i) If $E$ has signature $(s,t)$ at a point $e\in E\bs M$ 
then the Levi form of the hypersurface $\Sigma_{\phi(e)}$ has signature 
$(t+r-1,s)$ at $e$ from the side $\{\phi<\phi(e)\}$. 

(ii)
If $E$ has signature $(s,t)$ then the Levi form of $\phi$ has
signature $(t+r,s)$ (and hence $\phi$ is $(m-t+1)$-convex) on $E\bs M$,
and the Levi form of $\frac{1}{\phi}$ has signature $(s+1,t+r-1)$
(and hence $\frac{1}{\phi}$ is $(n-s)$-convex) on $E\bs M$. 

(iii) In particular, if $E$ is positive then $\frac{1}{\phi}$ is $r$-convex on $E\bs M$, 
and if $E$ is negative then $\phi$ is strongly plurisubharmonic on $E\bs M$. 
\end{proposition}

\begin{proof}  
Fix $e_0\in E\bs M$ and let $\Sigma=\Sigma_{\phi(e_0)}$. 
Choose local holomorphic coordinates 
$z=(z^1,\ldots,z^m)$ at $x_0=\pi(e_0)$ and a local holomorphic frame 
$(e_\rho)$ that is special at $x_0$. 
Then $e=\sum_{\sigma=1}^r \xi^\sigma e_\sigma$, 
$e_0=\sum_{\rho=1}^r \xi_0^\rho e_\rho(x_0)$, and 
$\phi(e)= \sum_{\rho,\sigma=1}^r h_{\rho\sigma} \xi^\sigma\bar \xi^\rho$.
Using (\ref{eq:frame}), a simple calculation 
\cite[p.\ 426]{Griffiths66} gives
\begin{align*}
     \di\dibar \phi(e_0) &=  \di_\xi\dibar_\xi \big|_{\xi=\xi_0} \sum_{\rho=1}^r \xi^\rho\bar\xi^\rho 
     + \sum_{\rho,\sigma=1,\ldots,r} \di_z\dibar_z h_{\rho\sigma}(x_0) \xi^\sigma \bar\xi^\rho \cr
     &= \sum_{\rho=1}^r d\xi^\rho\wedge d\bar\xi^\rho 
     - \sum_{\nad{\rho,\sigma=1,\ldots,r}{i,j=1,\ldots,m}} 
     \Theta^\rho_{\sigma i j}(x_0) 
     	\xi^\sigma \bar\xi^\rho dz^i\wedge d\bar z^j  \cr
     &= \sum_{\rho=1}^r d\xi^\rho\wedge d\bar\xi^\rho 
        - \sum_{i,j=1,\ldots,m} A_{ij}(x_0,\xi)\, dz^i \wedge d\bar z^j.
\end{align*}
The maximal complex tangent space to $\Sigma$ at $e_0$ consists of the vectors 
$\gamma=(\zeta^1,\ldots,\zeta^r;\eta^1,\ldots,\eta^m)$ with 
$\sum_{\rho=1}^r \xi^\rho_0 \bar \zeta^\rho=0$. In the $\zeta$-direction
(tangential to $E_{x_0}$) we thus get $r-1$ positive 
Levi eigenvalues for $\Sigma$;
in the $\eta$-direction (the horizontal direction in $T_{e_0}E$ with 
respect to the Chern connection) we get $s(e_0)$ negative and $t(e_0)$  positive
eigenvalues. 
Hence the Levi signature of $\Sigma$ at $e_0$ is $(t(e_0)+r-1,s(e_0))$.
The remaining Levi eigenvalue of $\phi$ in the radial direction is positive. 
All claims follow immediately.
\end{proof}

%
%
%
%
%
\section{Subvarieties in complements of submanifolds}
\label{subvariety-complements} 
In this section we combine our analytic techniques with
the differential geometric information from \S\ref{vector-bundles} 
to study the existence of subvarieties as in Theorem \ref{Main1} 
in total spaces of Hermitian holomorphic vector bundles, 
and in complements of certain compact complex submanifolds.

%
%
%
%
\begin{theorem}
\label{total-space}
Let $E\to M$ be a holomorphic vector bundle with fiber $\C^r$
over a compact $m$-dimensional complex manifold $M$. Set $n=m+r=\dim E$
and identify $M$ with the zero section of $E$.
Assume that $D\Subset S$ is a smoothly bounded, 
strongly pseudoconvex domain in a 
$d$-dimensional Stein manifold $S$ and $f_0\colon\bar D\to E$ is 
a continuous map that is holomorphic in $D$ and such that  
$f_0(bD) \subset E\bs M$. Then the following hold:

{\rm (i)} If $E$ admits a Hermitian metric with signature $(s,t)$
and $d<t+r$, then $f_0$ can be approximated uniformly on compacts 
in $D$ by proper holomorphic maps $f\colon D\to E$ such that
$f_0^{-1}(M)=f^{-1}(M)$. In particular, if $E$ is positive 
then this holds when $d<r={\rm rank}E$, and if $E$ is negative
then it holds when $d<n=\dim E$.

{\rm (ii)} If $E$ admits a Hermitian metric with signature $(s,t)$
with $d\le s$, then $f_0$ can be approximated uniformly on compacts 
in $D$ by holomorphic maps $f\colon D\to E$ such that
$f_0^{-1}(M)=f^{-1}(M)$ and the cluster set of $f$ at $bD$
belongs to the zero section $M$. In particular, if $E$ is positive 
then the above holds when $d\le m=\dim M$.
\end{theorem}

Recall that the {\em cluster set} of a map $f\colon D\to E$
at $bD$ is the set of all sequential limits 
$\lim_{j\to\infty} f(p_j)\in E$ along sequences 
$\{p_j\} \subset D$ without accumulation points in $D$
(therefore tending to $bD$).

\begin{proof}
Part (i) follows from Theorem \ref{Main1} and 
Proposition \ref{signature} (i), applied to $X=E$, $\rho=\phi$
and $\Omega=E\bs M$. (Note that $\phi$ is noncritical on $E\bs M$.)
To ensure that $f_0^{-1}(M)=f^{-1}(M)$,
it suffices to construct $f$ such that it agrees with $f_0$
to sufficiently high order at the finite set of points
in $f_0^{-1}(M)\subset D$. (If $d>r$ then $f_0(bD)\cap M=\emptyset$ 
implies $f_0(D)\cap M=\emptyset$ since the analytic set $f_0^{-1}(M) \subset D$, 
if nonempty, would have positive dimension.) 

To prove Part (ii), choose $c>0$ such that $f_0(\bar D) \subset W_c=\{\phi<c\}$
and apply Theorem \ref{Main1} with $\Omega =W_c\bs M$ and
$\rho=\frac{1}{\phi}-\frac{1}{c}$. 
\end{proof}

The situation in Theorem \ref{total-space} is a special
case of the following one.

\begin{theorem}
\label{complements}
Let $A$ be a compact complex submanifold in a complex manifold $X$
whose normal bundle $N_{A|X}$ has signature $(s,t)$ 
with respect to some Hermitian metric. There is an open
tubular neighborhood $V\subset X$ of $A$ with the following property.
If $D$ is a relatively compact, smoothly bounded, 
strongly pseudoconvex domain in a Stein manifold
with $\dim D \le s$, $z_0$ is a point in $D$ and $f_0\colon\bar D\to X$ is 
a continuous map that is holomorphic in $D$ and such that  
$f_0(bD) \subset V\bs A$, then $f_0$ can be approximated uniformly
on $D$ by holomorphic maps $f\colon D\to X$ such that
$f^{-1}(A)=f^{-1}_0(A)$, $f(z_0)=f_0(z_0)$, and the cluster set 
of $f$ at $bD$ belong to $A$. If $N_{A|X}$ is positive 
then the above holds when $\dim D \le \dim A$.
\end{theorem}

\begin{proof} 
Schneider proved in \cite{Schneider} that there exist a 
neighborhood $V\subset X$ of $A$ and a smooth
function $\rho\colon V\bs A \to \R$ that tends to $+\infty$ 
at $A$, and whose Levi form $\cL_\rho$ has at least 
$s+1$ positive eigenvalues at every point of $V\bs A$. 
We recall Schneider's construction to see that his function
is also noncritical on a deleted tubular neighborhood of $A$.
Once this is clear, it remains to apply Theorem \ref{Main1}.

Assume first that $A$ is a smooth complex hypersurface in $X$.
Let $E\to X$ denote the hyperplane section bundle of the divisor 
determined by $A$. Then $E|_A\simeq N_{A|X}$, and there is a holomorphic section
$\sigma \colon X\to E$ such that $A=\{x\in X\colon \sigma(x)=0\}$.
Such $\sigma$ is given by a collection $(g_i)$ of holomorphic 
functions $g_i\colon U_i\to\C$ on an open covering $\{U_i\}$ of $X$ 
such that $\{g_i=0\}=A\cap U_i$ and
$dg_i\ne 0$ on $A\cap U_i$. The associated 1-cocycle 
$g_{ij}=\frac{g_i}{g_j}$ defines the line bundle $E\to X$.

The Hermitian metric of signature $(s,t)$ on the normal bundle 
$E|_A= N_{A|X}$ extends to a Hermitian metric $h$ on $E$. 
On $E|_{U_i}\simeq U_i\times \C$ the metric is given  
by a positive function $h_i \colon U_i \to (0,\infty)$.
Let $||\sigma||_h^2 \colon X\to[0,\infty)$ be the squared length 
of the section $\sigma\colon X\to E$.
(On $U_i$ this equals $h_i|g_i|^2$.) Schneider showed that for 
a sufficiently large constant $C>0$ the metric $\phi$ on $E$, 
defined over $U_i$ by $\phi_i=\frac{h_i}{1+Ch_i|g_i|^2}$,
has signature $(s+1,t)$ over a neighborhood of $A$
(see \cite[p.\ 225]{Schneider}). 
Set $g=||\sigma||^2_\phi \colon X\to [0,\infty)$,  so 
\[
	g|_{U_i} = \phi_i |g_i|^2 = \frac{h_i|g_i|^2}{1+Ch_i|g_i|^2}
	=\frac{||\sigma||_h^2}{1+C||\sigma||_h^2}.
\]
It follows that
\[
	-i\di\dibar \log g|_{U_i} =  -i\di\dibar \log \phi_i  
\]
and hence the Levi form of $-\log g= -\log ||\sigma||^2_\phi$ 
has at least $s+1$ positive eigenvalues in a deleted neighborhood 
of $A$ in $M$. Clearly the same holds for 
$e^{-\log g}=\frac{1}{g}= \frac{1+C||\sigma||_h^2}{||\sigma||_h^2}$
and hence for $\rho = \frac{1}{||\sigma||_h^2}$. The latter function is 
noncritical near $A$ and it blows up along $A$.

The general case reduces to the hypersurface case by 
blowing up $X$ along $A$ \cite[\S 3]{Schneider}. 
Assume that $A$ has complex dimension $m$ and codimension $r$ in $X$.
Let $\hat A=\P(N)$ denote the total space of the fiber bundle over $A$ whose fiber
over a point $x \in A$ is $\P(N_x)\simeq \P^{r-1}$, 
the projective space of complex lines in 
$N_x\simeq\C^r$. Replacing $A$ by $\hat A$ changes $X$ 
to a new manifold $\hat X$ such that $\hat X\bs \hat A$ is biholomorphic 
to $X\bs A$, and $\hat A$ is a smooth complex hypersurface in $\hat X$.
The restriction of the normal bundle $N_{\hat A|\hat X}$ 
to the submanifold $\P(N_x)\subset \hat A$ 
is the universal bundle over $\P(N_x)\simeq \P^{r-1}$ 
(the inverse of the hyperplane section bundle).
This bundle is negative with respect to the 
Fubini-Study metric on $\P^{r-1}$, and a simple calculation shows 
that $N_{\hat A|\hat X}$ has signature $(s,t+r-1)$ if
$N_{A|X}$ has signature $(s,t)$. It remains to apply the previous argument 
(in the hypersurface case) to a deleted neighborhood 
of $\hat A$ in $\hat X$ (that is the same as 
a deleted neighborhood of $A$ in $X$).
\end{proof}

Assume now that $X$ is a complex manifold and $A$  
is a compact complex submanifold of $X$ with positive normal bundle 
$N_{A|X}$. Let $D$ be a relatively compact, 
smoothly bounded, strongly pseudoconvex domain 
in a Stein manifold, with $\dim D \le \dim A$.
Given a pair of points $z_0\in D$ and $x_0 \in X\bs A$, 
it is a natural question whether there exists a proper holomorphic map 
$f\colon D\to X\bs A$ such that $f(z_0)=x_0$.
Theorem \ref{complements} gives an affirmative answer
if $x_0$ lies sufficiently close to $A$ (take $f_0\colon \bar D\to X$ 
to be the constant map $f(z)=x_0$.) 

The answer to this question is negative in general. 
For example, if $X$ is obtained by
blowing up $\P^3$ at one point, with $\Sigma$ being the exceptional
divisor, and if $A\subset X$ is a complex hyperplane disjoint from $\Sigma$,
then for any 2-dimensional domain $z_0\in D\subset \C^2$ 
and holomorphic map $f\colon D\to X\bs A$ with $f(z_0)\in \Sigma$
the set $f^{-1}(\Sigma)$ is a positive dimensional subvariety
of $D$ which accumulates on $bD$; hence $f$ cannot be
proper into $X\bs A$. (We wish to thank the referee 
for pointing out this example.)

The situation is different when a connected topological group acts
transitively on $X$ by holomorphic automorphisms:

\begin{theorem}
\label{homogeneous}
Assume that $X$ is a complex manifold and $A$  
is a compact complex submanifold of $X$ with positive normal bundle 
$N_{A|X}$. Let $D$ be a relatively compact, smoothly bounded, 
strongly pseudoconvex domain in a Stein manifold, 
with $\dim D \le \dim A$ and $\dim A + \dim D < \dim X$.
Assume that a connected topological group $G$ acts transitively 
on $X$ by holomorphic automorphisms of $X$.
Given points $z_0\in D$ and $x_0\in X\bs A$, 
there exists a proper holomorphic map $f\colon D \to X\bs A$ 
such that $f(z_0)=x_0$.
\end{theorem}

\begin{proof}
The idea is taken from \cite[proof of Theorem 1]{FGl}. 
Let $s={\rm dim}A$. By the proof of Theorem \ref{complements} there 
exist a neighborhood $V\subset X$ of $A$ and a noncritical smooth
function $\rho\colon V\bs A \to \R$ that tends to $+\infty$ 
at $A$ and whose Levi form has at least $s+1$ positive eigenvalues 
on $V\bs A$.

Choose an open subset $V_0\subset X$ such that $A\subset V_0\Subset V$.
There is an open neighborhood $U$ of the identity in $G$ such that 
$g(V_0) \subset V$ for each $g\in U$. 

Choose a continuous map $f_0\colon \bar D\to X\bs A$ 
that is holomorphic in $D$ and such that $f_0(bD)\subset V_0$.
($f_0$ may be a constant map.) Choose $g\in G$ such that
$g f_0(z_0)=x_0$. As $G$ is connected, there are elements 
$g_1,\ldots,g_k\in U \subset G$ such that 
$g=g_k g_{k-1}\cdots g_1$ (product in $G$). 

The map $g_1 f_0$ is continuous on $\bar D$, holomorphic in $D$, 
and $g_1 f_0(bD)\subset V$. By general position we may assume
that its range does not intersect $A$.
Lemma \ref{bigstep} furnishes a continuous map $f_1\colon \bar D\to X\bs A$ 
that is holomorphic in $D$ such that $f_1(bD)\subset V_0$ and
$f_1(z_0)=g_1 f_0(z_0)$. (Lemma \ref{bigstep} also 
applies to individual maps in view of Lemma \ref{sprays-exist}.)

The map $g_2 f_1$ is continuous on $\bar D$, holomorphic in $D$, 
and $g_2 f_1(bD)\subset V$. By general position we 
may assume that its range does not intersect $A$. Lemma \ref{bigstep} furnishes 
a map $f_2\colon \bar D\to X\bs A$ that is holomorphic in $D$ 
such that $f_2(bD)\subset V_0$ and $f_2(z_0)=g_2 f_1(z_0)= g_2 g_1 f_0(z_0)$.

After $k$ such steps we obtain a continuous map 
$f_k\colon \bar D\to X\bs A$ that is holomorphic in $D$ and satisfies   
\[
	f_k(bD)\subset V_0,\quad 
	f_k(z_0)=g_k\cdots g_1 f_0(z_0)=x_0. 
\]
We now apply Theorem \ref{Main1} to $f_k$, with $\Omega= V_0\bs A$ 
and $\rho$ as above, 
to obtain a holomorphic map $f \colon D\to X\bs A$ 
such that $f(z_0)=x_0$ and the cluster set of $f$ at $bD$
belongs to $A$. Hence $f$ is a proper map of $D$ to $X\bs A$.  
\end{proof}

Applying Theorem \ref{homogeneous} with $X=\P^n$
and taking into account that every closed complex
submanifold of $\P^n$ has positive normal bundle
(see Barth \cite{Barth}) gives the following corollary. 
(Compare with Corollary \ref{proj-minus}.)

\begin{corollary}
\label{projective}
Let $D$ be a relatively compact, smoothly bounded, 
strongly pseudoconvex domain in a Stein manifold,
and let $A$ be a closed complex submanifold of a projective space $\P^n$
such that $\dim D\le \dim A$ and $\dim A + \dim D < n$.  
For every pair of points $z_0\in D$ and $x_0\in \P^n \bs A$ there exists
a proper holomorphic map $f\colon D \to \P^n\bs A$ with $f(z_0)=x_0$.
\end{corollary}

%
%
%
%
%

\medskip
\textit{Acknowledgement.}
We wish to thank  the referee for very helpful remarks
and suggestions which helped us to improve the presentation.

\bibliographystyle{amsplain}

\end{document}